\documentclass[11pt,reqno]{amsart}
 \usepackage[margin=1in]{geometry} 
\usepackage{amsmath,amsthm,amssymb,amsfonts,graphicx, fullpage, amsmath, url, verbatim, amssymb}

\usepackage{bbm}
\usepackage{color}
\usepackage{afterpage}

\newcommand{\R}{\mathbb{R}}

\newcommand{\al}{\alpha}

\newcommand{\pa}{\partial}

\numberwithin{equation}{section}
\newtheorem{theorem}{Theorem}[section]
\newtheorem{lemma}[theorem]{Lemma}
\newtheorem{corollary}[theorem]{Corollary}
\newtheorem{proposition}[theorem]{Proposition}

\newtheorem{remark}[theorem]{Remark}

\newtheorem{question}{Question}

\def\XXint#1#2#3{{\setbox0=\hbox{$#1{#2#3}{\int}$ }
\vcenter{\hbox{$#2#3$ }}\kern-.6\wd0}}

\definecolor{purple}{rgb}{0.8, 0, 1}
\definecolor{orange}{rgb}{1,.5,0}
\definecolor{purple}{rgb}{0.5, 0, 0.9}
\definecolor{green}{rgb}{0, 0.7, 0}
\definecolor{orange}{rgb}{1,.5,0}
\definecolor{gray}{rgb}{.6,.6,.6}

\begin{document}

\title{Remarks on stationary and uniformly-rotating vortex sheets: Rigidity results}

\author{Javier G\'omez-Serrano, Jaemin Park, Jia Shi and Yao Yao}

\begin{abstract}

In this paper, we show that the only  solution of the vortex sheet equation, either stationary or uniformly rotating with negative angular velocity $\Omega$, such that it has  positive vorticity and is  concentrated in a finite disjoint union of smooth curves with finite length is the trivial one: constant vorticity amplitude supported on a union of nested, concentric circles. The proof follows a desingularization argument and a calculus of variations flavor.

\vskip 0.3cm


\end{abstract}

\maketitle

\section{Introduction}

A vortex sheet is a weak solution of the 2D Euler equations:
\begin{equation}\label{eqs:euler}
v_t + v \cdot \nabla v = -\nabla p, \quad \nabla \cdot v = 0,
\end{equation}
 whose vorticity $\omega = \text{curl}(v)$ is a delta function supported on a curve or a finite number of curves $\Gamma_i = z_i(\al,t)$, i.e.
 \begin{equation}\label{omega_def} \omega(x,t) = \sum_{i} \varpi_i(\al,t) \delta(x - z_i(\al,t)).
\end{equation}
 Here  $\varpi_i(\al,t)$ is the vorticity strength on $\Gamma_i$ with respect to the  parametrization $z_i$, and the above equation is understood in the sense that 
 \[
 \int_{\mathbb{R}^2} \varphi(x) d\omega(x,t) = \sum_{i} \int \varphi(z_i(\alpha,t))\varpi_i(\alpha,t) d\alpha
 \]
  for all test functions $\varphi(x)\in C_0^\infty(\mathbb{R}^2)$.

The motivation of the study of the equation \eqref{eqs:euler} with vortex sheet initial data comes from the fact that in fluids with small viscosity, flows separate from rigid walls and corners \cite{Majda-Bertozzi:vorticity-incompressible-flow,Saffman:book-vortex-dynamics}. To model it, one may think of a solution to \eqref{eqs:euler} with one incompressible fluid where the velocity changes sign in a discontinuous (tangential) way across a streamline $z$. This discontinuity induces vorticity in $z$.

 The equations of motion of $\varpi_i$ and $z_i$ can be derived by means of the Birkhoff-Rott operator (\cite{Castro-Cordoba-Gancedo:naive-vortex-sheet,Lopes-Nussenzveig-Schochet:vortex-sheets-BR-formulation,Majda-Bertozzi:vorticity-incompressible-flow,Sulem-Sulem-Bardos-Frisch:instability-vortex-sheet}), namely:
\begin{equation}\label{def_BR0}
BR(z,\varpi)(x,t) =\frac{1}{2\pi}
PV\int \frac{(x-z(\beta,t))^{\bot}}{|x-
z(\beta,t)|^2}\varpi(\beta,t)d\beta,
\end{equation}
yielding
\begin{align}
\partial_t z_i(\al,t) & = \sum_j BR(z_j,\varpi_i)(z_i(\al,t)) + c_i(\al,t) \pa_\al z_i(\al,t) \label{eqs:vsheet1} \\
\partial_t \varpi_i(\al,t) & = \pa_\al(c_i(\al,t) \varpi_i(\al,t)), \label{eqs:vsheet}
\end{align}
where the term $c_i(\al,t)$ accounts for the reparametrization freedom of the curves. See the paper \cite{Izosimov-Khesin:vortex-sheets} by Izosimov--Khesin where they propose geodesic,  group-theoretic,  and Hamiltonian frameworks for their description.

The main goal of this paper is to establish radial symmetry properties of stationary/uniformly-rotating vortex sheets to \eqref{eqs:euler}. To do so, we first define what we mean by a stationary vortex sheet. Assume the initial data $\omega_0$ of \eqref{omega_def} is supported on a finite number of curves parametrized by $z_i(\alpha)$, with strength $\varpi_i(\alpha)$ (with respect to the parametrization $z_i$) respectively. If there exists some reparametrization choice $c_i(\alpha)$ such that the right hand sides of \eqref{eqs:vsheet1}--\eqref{eqs:vsheet} are both identically zero for every $i$, it gives that $\omega(\cdot,t)$ is invariant in time, and we say $\omega(\cdot,t)=\omega_0$ is a \emph{stationary vortex sheet}.

For any $x\in \mathbb{R}^2$ and $\Omega\in\mathbb{R}$, let $R_{\Omega t} x$ denote the rotation of $x$ counter-clockwise by angle $\Omega t$ about the origin. We say $\omega(x,t)=\omega_0(R_{\Omega t} x)$ is a \emph{uniformly-rotating vortex sheet} with angular velocity $\Omega$ if $\omega_0$ is stationary in the rotating frame with angular velocity $\Omega$. (Note that in the special case $\Omega=0$, the uniformly-rotating sheet is in fact stationary.) In Lemma~\ref{lemma_br_eq}, we will derive the equations satisfied by a stationary/rotating vortex sheet.

It is easy to see that if the $z_i$'s are concentric circles with constant $\varpi_i$ (with respect to the constant-speed parametrization) for every $i$, the solution is stationary, and it is also uniformly-rotating with any $\Omega\in\mathbb{R}$. We would like to understand the reverse implication, namely:

\begin{question}
Under what conditions must a stationary/uniformly-rotating vortex sheet be radially symmetric?
\end{question}

This type of rigidity question has been very lately understood for different equations and different settings such as in the papers by Koch--Nadirashvili--Sverak \cite{Koch-Nadirashvili-Seregin-Sverak:liouville-navier-stokes} for Navier-Stokes, Hamel--Nadirashvili \cite{Hamel-Nadirashvili:rigidity-euler-annulus,Hamel-Nadirashvili:shear-flow-euler-strip-halfspace,Hamel-Nadirashvili:liouville-euler} for the 2D Euler equation on a strip, punctured disk or the full plane, G\'omez-Serrano--Park--Shi--Yao \cite{GomezSerrano-Park-Shi-Yao:radial-symmetry-stationary-solutions} for the 2D Euler and modified SQG in the full plane  and Constantin--Drivas--Ginsberg \cite{Constantin-Drivas-Ginsberg:rigidity-flexibility} for the 2D and 3D Euler, as well as the 2D Boussinesq and the 3D Magnetohydrostatic (MHS) equations.

The next theorem is the main result of the paper, solving it for the vortex sheet equations:

\begin{theorem}\label{thmA}
Let $\omega(x,t)=\omega_0(R_{\Omega t} x)$ be a stationary/uniformly-rotating vortex sheet with angular velocity $\Omega$. Assume that $\omega_0$ is concentrated on $\Gamma$, which is a finite union of smooth curves, and $\omega_0$ has positive vorticity strength on $\Gamma$. (See \textbf{\textup{(H1)}}--\textbf{\textup{(H3)}} in Section~\ref{sec2} for the precise regularity and positivity assumptions.) 

If $\Omega\leq 0$, $\Gamma$ must be a union of concentric circles, and $\omega_0$ must have constant strength along each circle (with respect to the constant-speed parametrization). In addition, if $\Omega<0$, all circles must be centered at the origin.
\end{theorem}

We now go first over the history of the equations \eqref{eqs:vsheet1}--\eqref{eqs:vsheet}, focusing later on the case of steady solutions. The study of those solutions is important due to the ill-posedness of the vortex sheet equation, thus they represent (unstable) structures for which there is global existence.

\subsection{Brief history of the dynamical problem}

The existence of solutions to \eqref{eqs:vsheet1}--\eqref{eqs:vsheet} has been widely studied. The seminal paper of Delort \cite{Delort:vortex-sheet} proved global existence of weak solutions of \eqref{eqs:euler} for an initial velocity in $L^{2}_{loc}$ and a vorticity a positive Radon measure. Majda \cite{Majda:vortex-sheet} provided a simpler proof. See also the works by Schochet \cite{Schochet:weak-vorticity-formulation-2d-euler,Schochet:point-vortex-method-vortex-sheet}
 and Evans--Muller \cite{Evans-Muller:vortex-sheet}. All of them use the hypothesis that the vorticity has a definite sign.

If the vorticity does not have a sign, Lopes Filho--Nussenzveig Lopes--Xin proved existence in \cite{Lopes-Nussenzveig-Xin:vortex-sheets-reflection}, in the case where the system enjoys reflection symmetry. For the setting in which the curve $z_i$ is not closed and represented as a graph, Sulem--Sulem--Bardos--Frisch \cite{Sulem-Sulem-Bardos-Frisch:instability-vortex-sheet} proved local existence in the case of analytic initial data. 

The first sign of singularities with analytic initial data goes back to Moore \cite{Moore:singularity-vortex-sheet}, where he demonstrated that the curvature may blow up in finite time. Ebin \cite{Ebin:ill-posedness-vortex-sheet} showed ill-posedness in Sobolev spaces when $\gamma$ has a distinguished sign, and Duchon--Robert \cite{Duchon-Robert:global-vortex-sheet} proved global existence for a class of initial data in the unbounded setting. Caflisch--Orellana \cite{Caflisch-Orellana:singular-solutions-ill-posedness-vortex-sheet} also showed global existence for a class of initial data, as well as ill-posedness in $H^s$ for $s > \frac32$ and simplified the analysis of Moore \cite{Caflisch-Orellana:long-time-existence-vortex-sheet}. We also mention here the work of Wu \cite{Wu:vortex-sheet}, in which she proved the existence of solutions to \eqref{eqs:vsheet1}--\eqref{eqs:vsheet} in spaces which are less regular than $H^s$. Sz\'ekelyhidi \cite{Szekelyhidi:weak-solutions-vortex-sheet} (resp. Mengual--Sz\'ekelyhidi \cite{Mengual-Szekelyhidi:vortex-sheet}) constructed infinitely many admissible weak solutions to \eqref{eqs:euler} for vortex sheet initial data with (resp. without necessarily) a distinguished sign.

\subsection{Stationary and rotating solutions}

Relative equilibria are an important family of solutions of fluid equations since their structures persist for long times. This is specially important when the equations of motion are ill-posed. In the particular case of \eqref{eqs:vsheet1}--\eqref{eqs:vsheet}, our knowledge is very small and only very few explicit cases are known: the circle and the straight line (with constant $\gamma$), which are stationary, and the segment of length $2a$ and density 
\begin{equation}\label{example_rotate}
\gamma(x) = \Omega \sqrt{a^2 - x^2}, \qquad x \in [-a,a],
\end{equation}
which is a rotating solution with angular velocity $\Omega$ \cite{Batchelor:book-fluid-dynamics}. 
Protas--Sakajo \cite{Protas-Sakajo:rotating-equilibria-vortex-sheet} generalized this solution and proved the existence of several others made out of segments rotating about a common center of rotation with endpoints at the vertices of a regular polygon by solving a Riemann-Hilbert problem, even finding some of them analytically.

In the paper \cite{GomezSerrano-Park-Shi-Yao:rotating-solutions-vortex-sheet} we prove the existence of a family of vortex sheet rotating solutions with non-constant vorticity density supported on a non-radial curve, bifurcating from the circle with constant density.

Numerically, some solutions have been computed before. O'Neil \cite{ONeil:relative-equilibria-vortex-sheets,ONeil:collapse-vortex-sheets} used point vortices to approximate the vortex sheet and compute uniformly rotating solutions and Elling \cite{Elling:vortex-sheet-cusps} constructed numerically self-similar vortex sheets forming cusps. O'Neil \cite{ONeil:point-vortices-vortex-sheets,ONeil:point-vortices-vortex-sheets-pof} also found numerically steady solutions which are combinations of point vortices and vortex sheets.

\subsection{Structure of the proof}

The proof is inspired by our recent rigidity result in the paper \cite{GomezSerrano-Park-Shi-Yao:radial-symmetry-stationary-solutions} on stationary and rotating solutions of the 2D Euler equations both in the smooth and vortex patch settings. To prove it, we constructed an appropriate functional and showed, on one hand, that any stationary solution had to be a critical point, and on the other, for any curve which is not a circle there existed a vector field along which the first variation was non-zero. This vector field is defined in terms of an elliptic equation in the interior of the patch. In the case of the vortex sheet, this is not possible anymore. Instead, we desingularize the problem by considering patches of thickness $\sim\varepsilon$ which are tubular neighborhoods of the sheet. The drawback is that we lose the property that any stationary solution has to be a critical point if $\varepsilon > 0$ and very careful, quantitative estimates need to be done to show that indeed the first variation of a stationary solution tends to 0 as $\varepsilon \to 0$. This setup is also reminiscent of the numerical work by Baker--Shelley \cite{Baker-Shelley:vortex-sheets-via-vortex-patches}, where they approximate the motion of a vortex sheet by a vortex patch of very small width. In \cite{Benedetto-Pulvirenti:vortex-layers-vortex-sheet}, Benedetto--Pulvirenti proved the stability (for short time) of vortex sheet solutions with respect to solutions to 2D Euler with a thin strip of vorticity around a curve. See also the work by Caflisch--Lombardo--Sammartino \cite{Caflisch-Lombardo-Sammartino:vortex-layers-sheets} for more stability results with a different desingularization.

\subsection{Organization of the paper}
In Section \ref{sec2} the equations for the stationary/rotating vortex sheet are derived, and in Section \ref{sec3} we perform the desingularization procedure. Section \ref{subsec_p} is devoted to construct the aforementioned divergence free vector-field along which the first variation is non-zero. Finally in Section \ref{sec5} we conclude the quantitative estimates and prove the symmetry result from Theorem \ref{thmA}.

\subsection{Notations}
For a bounded domain $D\subset \mathbb{R}^2$, we denote $|D|$ by its area (i.e. its Lebesgue measure). For $x\in \mathbb{R}^2$ and $r>0$, denote by $B(x,r)$ or $B_r(x)$ the open ball centered at $x$ with radius $r$.

Through Section 3-5 of this paper, we will desingularize the vortex sheet into a vortex layer with width $\sim\epsilon$, and obtain various quantitative estimates. In all these estimates, we say a term $f$ is $O(g(\epsilon))$ if $|f|\leq Cg(\epsilon)$ for some constant $C$ independent of $\epsilon$.

For a domain $U\subset \mathbb{R}^2$, in the boundary integral $\int_{\partial U}\vec{f}\cdot n d\sigma$, $n$ denotes the outer normal of the domain $U$.

\section{Equations for a stationary/rotating vortex sheet}\label{sec2}
Let $\omega(\cdot,t)=\omega_0(R_{\Omega t})$ be a stationary/rotating vortex sheet solution to the incompressible 2D Euler equation, where $ \omega_0\in \mathcal{M}(\mathbb{R}^2) \cap H^{-1}(\mathbb{R}^2)$ is a Radon measure. Here $\Omega=0$ corresponds to a stationary solution, and $\Omega\neq 0$ corresponds to a rotating solution.  Assume $\omega_0$ is concentrated on $\Gamma$, which is a finite disjoint union of curves. Throughout this paper we assume $\Gamma$ satisfies the following:

\textbf{(H1)} Each connected component of $\Gamma$ is smooth and with finite length, and it is either a simple closed curve (denote them by $\Gamma_1,\dots,\Gamma_n)$, or a non-self-intersecting curve with two endpoints (denote them by $\Gamma_{n+1},\dots, \Gamma_{n+m}$). Here we require $n+m\geq 1$, but allow either $n$ or $m$ to be 0.  

Let us denote 
\begin{equation}
\label{def_d_gamma}
d_{\Gamma}:= \min_{k\neq i} \text{dist}(\Gamma_i, \Gamma_k),
\end{equation}
 which is strictly positive since we assume the curves $\{\Gamma_i\}_{i=1}^{n+m}$ are disjoint.
 For $i=1,\dots, n+m$, denote by $L_i$ the length of $\Gamma_i$. Let $z_i:S_i\to \Gamma_i$ denote a constant-speed parameterization of $\Gamma_i$ (in counter-clockwise direction if $\Gamma_i$ is a closed curve), where the parameter domain $S_i$ is given by
\[
S_i:=\begin{cases}\mathbb{R}/\mathbb{Z}& \text{for $i=1,\dots,n$},\\
[0,1] &\text{for $i=n+1,\dots, n+m$}.
\end{cases}
\] 
Note that this gives $|z_i'|\equiv L_i$, and the arc-chord constant
\begin{equation}\label{def_arc_chord}
F_\Gamma :=  \max_{i=1,\dots,n+m}\sup_{\alpha\neq \beta\in S_i}\frac{|\alpha-\beta|}{|z_i(\alpha)-z_i(\beta)|}
\end{equation}
is finite,
since $\Gamma$ is non-self-intersecting. 
Let $\mathbf{s}:\Gamma\to\mathbb{R}^2$ be the unit tangential vector on $\Gamma$, given by $\mathbf{s}(z_i(\alpha)) := \frac{z_i'(\alpha)}{|z_i'(\alpha)|} = \frac{z_i'(\alpha)}{L_i}$, and $\mathbf{n}:\Gamma\to\mathbb{R}^2$ be the unit normal vector, given by $\mathbf{n} = \mathbf{s}^\perp$.  See Figure~\ref{fig1} for an illustration.

For $i=1,\dots, n+m$, let us denote by $\gamma_i(\alpha)$ the vorticity strength at $z_i(\alpha)$ with respect to the arclength parametrization, which is related to $\varpi_i(\alpha)$ by 
\begin{equation}\label{def_gamma}
\gamma_i(\alpha) = \frac{\varpi_i(\alpha)}{| z_i'(\alpha)|}  \quad\text{ for }\alpha\in S_i.
\end{equation}
Throughout this paper we will be working with $\gamma_i$, instead of $\varpi_i$. 
We impose the following regularity and positivity assumptions on $\gamma_i$:

\textbf{(H2)} Assume that $\gamma_i \in C^2(S_i)$ for $i=1,\dots,n$ and $\gamma_i \in C^{b}(S_i) \cap C^1(S_i^\circ)$ for some $b\in (0,1)$ for \color{black}  $i=n+1,\dots,n+m$.\footnote{For an open curve $i=n+1,\dots,n+m$, note that \textbf{(H2)} does not require $\gamma_i$ to be $C^1$ up to the boundary of $S_i$, and its derivative is allowed to blow up at the endpoints.  This is motivated by the fact that in the explicit uniformly-rotating solution \eqref{example_rotate}, its strength $\gamma$ is H\"older continuous in $[-a,a]$ and smooth in the interior, but its derivative blows up at the endpoints.} 

\textbf{(H3)} For $i=1,\dots,n$, assume $\gamma_i > 0$ in $S_i$. And for $i=n+1,\dots,n+m$, assume $\gamma_i > 0$ in $S_i^\circ$, and $\gamma_i(0)=\gamma_i(1)=0$.

Note that for a closed curve, \textbf{(H3)} implies that $\gamma_i$ is uniformly positive; whereas for an open curve, $\gamma_i$ is positive in the interior of $S_i$ but vanishes at its endpoints. This is because any stationary/rotating vortex sheet with continuous $\gamma_i$ must have it vanishing at the two endpoints of any open curve: if not, one can easily check that $|BR(z_i(\alpha))\cdot\mathbf{n}(z_i(\alpha))|\to\infty$ as $\alpha$ approaches the endpoint, thus such a vortex sheet cannot be stationary in the rotating frame.

With the above notations of $z_i$ and $\gamma_i$, the Birkhoff-Rott integral \eqref{def_BR0} along the sheet can now be expressed as
\begin{equation}\label{def_BR}
BR(z_i(\alpha)) = \sum_{k=1}^{n+m} BR_k(z_i(\alpha)) := \sum_{k=1}^{n+m}PV \int_{S_k}  K_2(z_i(\alpha) - z_k(\alpha'))\, \gamma_k(\alpha') |z_k'(\alpha')| \,d\alpha',
\end{equation}
with the kernel $K_2$ given by 
\begin{equation}\label{def_k2}K_2(x) := (2\pi)^{-1} \nabla^\perp \log|x| = \frac{x^\perp}{2\pi |x|^2},
\end{equation} and the principal value in \eqref{def_BR} is only needed for the integral with $k=i$.

Let $\mathbf{v}:\mathbb{R}^2\to\mathbb{R}^2$ be the velocity field generated by $\omega_0$, given by  $\mathbf{v}:= \nabla^\perp (\omega_0 * \mathcal{N})$. Note that $\mathbf{v} \in C^\infty(\mathbb{R}^2 \setminus \Gamma)$, but $\mathbf{v}$ is discontinuous across $\Gamma$. Let $\mathbf{v}^+,\mathbf{v}^-: \Gamma\to\mathbb{R}^2$ 
denote the two limits of $\mathbf{v}$ on the two sides of $\Gamma$ (with $\mathbf{v}^+$ being the limit on the side that $\mathbf{n}$ points into -- see Figure~\ref{fig1} for an illustration)\color{black}, and $[\mathbf{v}] := \mathbf{v}^- - \mathbf{v}^+$ the jump in $\mathbf{v}$ across the sheets. 
 $[\mathbf{v}]$ is related to the vortex-sheet strength $\gamma$ as follows (see \cite[Eq. (9.8)]{Majda-Bertozzi:vorticity-incompressible-flow} for a derivation): $[\mathbf{v}]\cdot \mathbf{n}=0$, and
\[
[\mathbf{v}] \times \mathbf{n} = [\mathbf{v}] \cdot \mathbf{s} = \gamma.
\]
In addition, the Birkhoff-Rott integral \eqref{def_BR} is the the average of $\mathbf{v}^+$ and $\mathbf{v}^-$, namely
\[
BR(z_i(\alpha)) =\frac{1}{2}(\mathbf{v}^+(z_i(\alpha))+\mathbf{v}^-(z_i(\alpha))) \quad\text{ for all }\alpha\in S_i, i=1,\dots,n+m.
\]

\begin{figure}[h!]
\begin{center}
\includegraphics[scale=1.2]{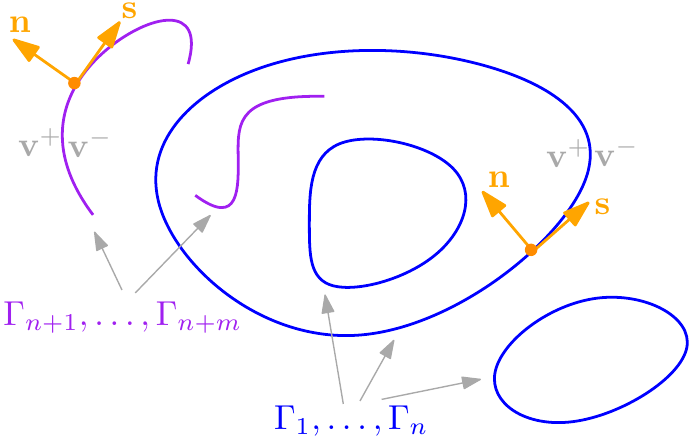}
\caption{Illustration of the closed curves $\Gamma_1,\dots,\Gamma_n$ and the open curves $\Gamma_{n+1},\dots,\Gamma_{n+m}$, and the definitions of $\mathbf{n}$, $\mathbf{s}$, $\mathbf{v}^+$ and $\mathbf{v}^-$.\label{fig1}}
\end{center}
\end{figure}

In the following lemma, we derive the equation that the Birkhoff-Rott integral satisfies for a stationary/rotating vortex sheet.

\begin{lemma}\label{lemma_br_eq}
Assume $\omega(\cdot,t)=\omega_0(R_{\Omega t} x)$ is a stationary/uniformly-rotating vortex sheet with angular velocity $\Omega\in\mathbb{R}$, and $\omega_0$ is concentrated on $\cup_{i=1}^{n+m}\Gamma_i$, with $z_i$ and $\gamma_i$ defined as above. Then the Birkhoff-Rott integral $BR$ \eqref{def_BR} and the strength $\gamma_i$ satisfy the following two equations:
\begin{equation}\label{BR1}
(BR-\Omega x^\perp)\cdot \mathbf{n} =\mathbf{v}^+ \cdot \mathbf{n} = \mathbf{v}^- \cdot \mathbf{n} = 0 \quad\text{ on } \Gamma,
\end{equation}
and
\begin{equation}\label{BR2}
(BR(z_i(\alpha))-\Omega z_i^\perp(\alpha)) \cdot \mathbf{s}(z_i(\alpha)) \,\gamma_i(\alpha)= \begin{cases} C_i  &\text{ on } S_i \text{ for } i = 1,\dots, n, \\0 & \text{ on } S_i \text{ for } i = n+1,\dots, n+m.
\end{cases}
\end{equation}
In particular, the above two equations imply that $BR(z_i(\alpha))-\Omega z_i^\perp(\alpha) \equiv\mathbf{0}$ for $i=n+1,\dots,n+m$.
\end{lemma}

\begin{proof}
By definition of the stationary/uniformly-rotating solutions,  $\omega_0$ is a stationary vortex sheet in the rotating frame with angular velocity $\Omega$. In this rotating frame, an extra velocity $-\Omega z_i^\perp$ should be added to the right hand side of \eqref{eqs:vsheet1}. Therefore the evolution equations \eqref{eqs:vsheet1}--\eqref{eqs:vsheet} become the following in the rotating frame (where we also use \eqref{def_BR}):
\begin{align}
\partial_t z_i(\al,t) & = BR(z_i(\al,t))  - \Omega z_i^\perp(\alpha,t)  + c_i(\al,t) \pa_\al z_i(\al,t) \label{eqs:vsheet10} \\
\partial_t \varpi_i(\al,t) & =\pa_\al(c_i(\al,t) \varpi_i(\al,t)), \label{eqs:vsheet0}
\end{align}
where the term $c_i(\al,t)$ accounts for the reparametrization freedom of the curves.
Since $\omega_0$ is stationary  in the rotating frame, $z_i(\cdot,t)$ parametrizes the same curve as $z_i(\cdot,0)$. Therefore $\partial_t z_i(\al,t)$ is tangent to the curve $\Gamma_i$, and multiplying $\mathbf{n}(z_i(\alpha,t))$ to \eqref{eqs:vsheet10} gives
\begin{equation}\label{temp_dot_n}
0=\partial_t z_i(\al,t) \cdot \mathbf{n}(z_i(\alpha,t)) = (BR(z_i(\alpha,t))-\Omega z_i^\perp(\alpha,t))\cdot \mathbf{n}(z_i(\alpha,t)),
\end{equation}
where we use that $\mathbf{n}(z_i(\alpha,t)) \cdot \partial_\alpha z_i(\alpha,t)=0$. This proves \eqref{BR1}.

Now we prove \eqref{BR2}.  
Towards this end, let us choose 
\[
c_i(\alpha,t) := -\frac{(BR(z_i(\al,t))-\Omega z_i^\perp(\alpha,t))\cdot \mathbf{s}(z_i(\al,t))}{|\pa_\al z_i(\al,t)|},
\] so that  multiplying $\mathbf{s}(z_i(\alpha,t))$ to \eqref{eqs:vsheet10} gives $\partial_t z_i(\al,t) \cdot \mathbf{s}(z_i(\alpha,t)) =0$, and combining it with \eqref{temp_dot_n} gives
$
\partial_t z_i(\alpha,t) =0.
$
In other words, with such choice of $c_i$, the parametrization $z_i(\alpha,t)$ remains fixed in time. Since $\omega_0$ is stationary in the rotating frame, we know that with a fixed parametrization $z_i(\alpha,t)=z_i(\alpha,0)$, the strength $\varpi_i(\alpha,t)$ must also remain invariant in time. Thus \eqref{eqs:vsheet0} becomes
\[
c_i(\alpha,t) \varpi_i(\alpha,t) \equiv C_i.
\]
Plugging the definition of $c_i$ into  the equation  above and using the fact that $z_i$ is invariant in $t$, we have
\[
\frac{(BR(z_i(\al))-\Omega z_i^\perp(\alpha))\cdot \mathbf{s}(z_i(\al)) \varpi_i(\alpha)}{|\pa_\al z_i(\al)|  }\equiv -C_i\quad\text{ for all }\alpha\in S_i,
\]
and finally the relationship between $\gamma_i$ and $\varpi_i$ in \eqref{def_gamma} yields \eqref{BR2} for $i=1,\dots, n$. 

And for the open curves $i=n+1,\dots,n+m$, note that we do not have any reparametrization freedom at the two endpoints $\alpha=0,1$, therefore the  endpoint  velocity $BR(z_i(0,t))-\Omega z_i^\perp(0,t)$ must be 0 to ensure that $\omega_0$ is stationary in the rotating frame. This immediately leads to $C_i=0$ for $i=n+1,\dots,n+m$, finishing the proof of \eqref{BR2}.
\end{proof}

\section{Approximation by a thin vortex layer}\label{sec3}
Our aim in this section is to desingularize the vortex sheet $\omega_0$. Namely, for $0<\epsilon\ll 1$, we will construct a vorticity $\omega^\epsilon \in L^\infty(\mathbb{R}^2) \cap L^1(\mathbb{R}^2)$ that only takes values $0$ and $\epsilon^{-1}$, and is supported in an $O(\epsilon)$ neighborhood of $\Gamma$, such that $\omega^\epsilon$ weakly converges to $\omega_0$ as $\epsilon\to 0^+$.

For each $i=1,\dots,n+m$, we will describe a neighborhood of $\Gamma_i$ using the following change of coordinates: let $R^\epsilon_i: S_i\times \mathbb{R} \to \mathbb{R}^2$ be given by
\begin{equation}\label{def_r_eps}
R^\epsilon_i(\alpha, \eta) := z_i(\alpha) + \epsilon \gamma_i(\alpha)  \mathbf{n}(z_i(\alpha)) \eta,
\end{equation}
and let
\[
D^\epsilon_i := \left\{ R^\epsilon_i(\alpha, \eta) : \alpha\in S_i^\circ, \eta\in (-1, 0)\right\}.
\]
Note that each $D^\epsilon_i$ is a connected open set, and for all $\epsilon>0$ sufficiently small, the sets $(D^\epsilon_i)_{i=1}^{n+m}$ are disjoint. For $i=1,\dots,n$, the domains $D_i^\epsilon$ are doubly-connected with smooth boundary, and its inner boundary coincides with $\Gamma_i$; see the left of Figure~\ref{fig_R} for an illustration. And for $i=n+1,\dots,n+m$, the domains $D_i^\epsilon$ are simply-connected, and its boundary is smooth except at at most two points; see the right of Figure~\ref{fig_R} for an illustration. 

\begin{figure}[h!]
\begin{center}
\includegraphics[scale=1]{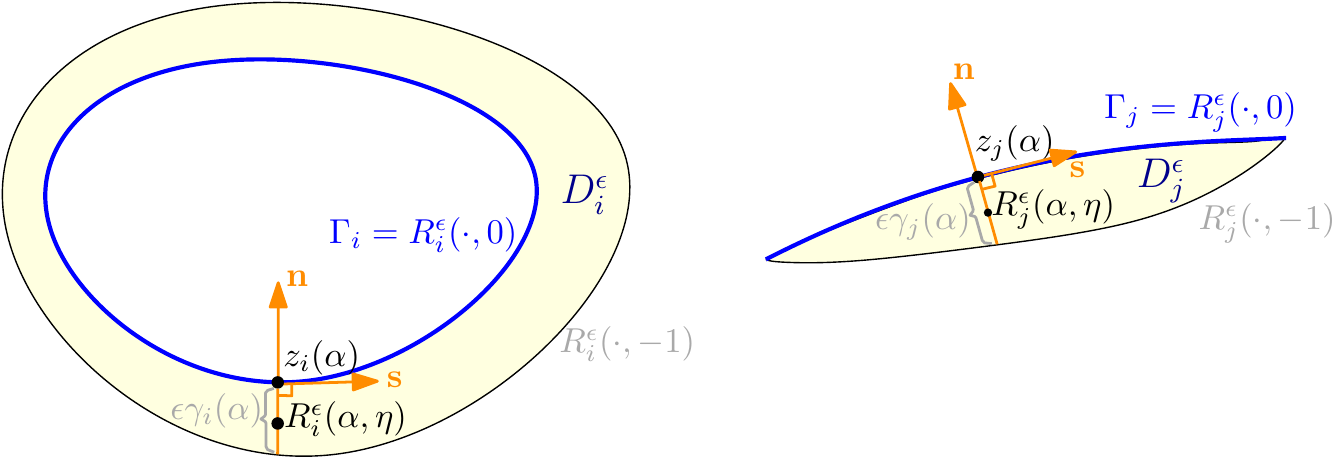}
\caption{Illustration of the definitions of $R_i^\epsilon$ and $D_i^\epsilon$ for a closed curve (left) and an open curve (right).\label{fig_R}}
\end{center}
\end{figure}

In addition, for $\epsilon>0$ that is sufficiently small, one can check that $R^\epsilon_i: S_i^\circ \times (-1,0) \to D_i^\epsilon$ is a diffeomorphism. Since $\gamma_i \in C^1(S_i)$ and $z_i\in C^2(S_i)$, we only need to show $R^\epsilon_i: S_i^\circ \times (-1,0) \to D_i^\epsilon$ is injective. Below we prove this fact in a stronger quantitative version, which will be used later.

\begin{lemma}\label{lemma_injection}
For any $i=1,\dots, n+m$, assume $\Gamma_i$ and $\gamma_i$ satisfy \textbf{\textup{(H1)}}--\textbf{\textup{(H2)}}. Then the map $R^\epsilon_i: S_i^\circ \times (-1,0) \to D^\epsilon_i$ given by \eqref{def_r_eps} is injective. In addition,  there exist some $c_0,\epsilon_0>0$ depending on $\|z_i\|_{C^2(S_i)},$ $\rVert \gamma_i\rVert_{L^\infty(S_i)}$ and $F_\Gamma$, such that for all $\epsilon\in(0,\epsilon_0)$ we have
\begin{equation}\label{r_diff}
|R^\epsilon_i(\alpha',\eta')-R^\epsilon_i(\alpha,\eta) |\geq c_0\big(|\alpha'-\alpha| + \epsilon |\gamma_i(\alpha)\eta - \gamma_i(\alpha')\eta'|\big),
\end{equation}
for all $\alpha,\alpha'\in S_i^\circ, ~\eta,\eta' \in (-1,0)$.\footnote{In fact, \eqref{r_diff} also holds (with a slightly smaller $\epsilon_0$ and $c_0$) for $\eta,\eta' \in (-2,2)$, even though such $R_i^\epsilon$ may not belong to $D_i^\epsilon$. We will use this fact later in the proof of Lemma~\ref{lemma_v_jump}.}
\end{lemma}

\begin{proof} To begin with, note that \eqref{r_diff} immediately implies that $R^\epsilon_i:S_i^\circ \times (-1,0) \to D^\epsilon_i$ is injective, where we used the positivity assumption $\gamma_i>0$ in $S_i^\circ$ in \textbf{(H2)}. Thus it suffices to prove \eqref{r_diff}. Throughout the proof, we fix any $i\in\{1,\dots, n+m\}$, and we will omit the subscript $i$ for notational simplicity. Using the definition \eqref{def_r_eps}, let us break $R^\epsilon(\alpha',\eta') - R^\epsilon(\alpha,\eta)$ into 
 \begin{equation}\label{temp001}
 R^\epsilon(\alpha',\eta') - R^\epsilon(\alpha,\eta) = \underbrace{z(\alpha') - z(\alpha)}_{=:T_1} +\underbrace{\epsilon\left( \gamma(\alpha')\eta' - \gamma(\alpha)\eta \right)\mathbf{n}(z(\alpha'))}_{=:T_2} + \underbrace{\epsilon\gamma(\alpha)\eta \left( \mathbf{n}(z(\alpha'))-\mathbf{n}(z(\alpha)) \right)}_{=:T_3}.
 \end{equation}
 For $T_1$ and $T_3$,  we have 
 \begin{equation}\label{T1T2}
 \begin{split}
 \left| T_1 - z'(\alpha')(\alpha'-\alpha)\right| &\le \rVert z \rVert_{C^2(S)}|\alpha-\alpha'|^2, \\
 \left| T_3 \right| &\le \epsilon\gamma(\alpha)\rVert z \rVert_{C^2(S)}|\alpha-\alpha'|.
 \end{split}
 \end{equation}
 Also, using that $z'(\alpha')=L \mathbf{s}(z(\alpha'))$ is perpendicular to $\mathbf{n}(z(\alpha'))$, we have
  \[
  \begin{split}
  |z'(\alpha')(\alpha'-\alpha) + T_2| &= \left| L(\alpha'-\alpha) \mathbf{s}(z(\alpha')) + \epsilon\left( \gamma(\alpha')\eta' - \gamma(\alpha)\eta \right)\mathbf{n}(z(\alpha'))\right| \\
  &\geq \frac{1}{2}L|\alpha'-\alpha| + \frac{1}{2}\epsilon\left| \gamma(\alpha')\eta' - \gamma(\alpha)\eta \right|,
  \end{split}
  \]
  where we use that $\sqrt{x^2+y^2}\geq \frac{1}{2}(|x|+|y|)$. Combining this with \eqref{T1T2} gives
 \begin{align*}
 | T_1 + T_2  + T_3| 
& \geq |\alpha - \alpha'| \left( \frac{L}{2} - \rVert z\rVert_{C^2(S)}\left( |\alpha-\alpha'| + \epsilon\gamma(\alpha) \right) \right) +\frac{1}{2} \epsilon | \gamma(\alpha')\eta' - \gamma(\alpha)\eta |,
 \end{align*}
 thus
 \begin{equation}\label{bound1_smallalpha}
 |R^\epsilon(\alpha',\eta') - R^\epsilon(\alpha,\eta)| \ge \frac{L}{4}|\alpha-\alpha'| + \frac{1}{2}\epsilon | \gamma(\alpha')\eta' - \gamma(\alpha)\eta |
 \end{equation}
 for all $0<\epsilon < L(8\rVert z \rVert_{C^2}\rVert\gamma\rVert_{L^\infty})^{-1}$ and $|\alpha-\alpha'| \le \frac{L}{8\rVert z \rVert_{C^2}}$. 
  
  For $|\alpha-\alpha'| > \frac{L}{8\rVert z \rVert_{C^2}}$, recall that the definition of $F_\Gamma$ in \eqref{def_arc_chord} gives $ |z(\alpha')-z(\alpha)| \geq F_\Gamma^{-1}|\alpha'-\alpha|$. Thus a crude estimate gives
  \begin{equation}\label{case2}
  |R^\epsilon(\alpha',\eta') - R^\epsilon(\alpha,\eta)|\geq |z(\alpha')-z(\alpha)| -2\epsilon \|\gamma\|_{L^\infty(S)} \geq \frac{1}{2F_\Gamma} |\alpha'-\alpha| + \epsilon | \gamma(\alpha')\eta' - \gamma(\alpha)\eta |
  \end{equation}
 for $0<\epsilon<L(64F_\Gamma\|z\|_{C^2}\|\gamma\|_{L^\infty})^{-1}$. (Note that for such $\epsilon$ we have $4\epsilon\|\gamma\|_{L^\infty} \leq \frac{1}{2F_\Gamma}|\alpha'-\alpha|$ due to our assumption that $|\alpha-\alpha'| > \frac{L}{8\rVert z \rVert_{C^2}}$).
 
 Finally, combining \eqref{bound1_smallalpha} and \eqref{case2}, it follows that \eqref{r_diff} holds for $c_0 = \min\{ \frac{L}{4}, \frac{1}{2F_\Gamma}, \frac{1}{2}\}$ and $\epsilon_0 = \min\{ L(8\rVert z \rVert_{C^2}\rVert\gamma\rVert_{L^\infty})^{-1}, L(64F_\Gamma\|z\|_{C^2}\|\gamma\|_{L^\infty})^{-1}\}$. This finishes the proof. 
 \color{black}
\end{proof}

In the next lemma we compute the partial derivatives and Jacobian of $R^\epsilon_i(\alpha,\eta)$, which will be useful later.

\begin{lemma}\label{P_derivatives1}
For any $i=1,\dots,n+m$, let $z_i$ be a constant-speed parameterization of the curve $\Gamma_i$ (with length $L_i$), and let $R^\epsilon_i$ be given by \eqref{def_r_eps}. Then its partial derivatives are
\begin{equation}\label{eq_partial}
\begin{split}
\partial_\alpha  R^\epsilon_i(\alpha,\eta) &= z_i'(\alpha) + \epsilon \left(\gamma_i'(\alpha) \frac{z_i'(\alpha)^\perp}{L_i} \eta + \gamma_i(\alpha) \frac{z_i''(\alpha)^\perp}{L_i} \eta \right),\\
\partial_\eta  R^\epsilon_i(\alpha,\eta) &=  \epsilon \gamma_i(\alpha)  \frac{z_i'(\alpha)^\perp}{L_i}.
\end{split}
\end{equation}
Moreover, its Jacobian is given by
\begin{equation}\label{eq_det}
\begin{split}
\det (\nabla_{\alpha,\eta} R^\epsilon_i) & = \epsilon L_i \gamma_i(\alpha) - \epsilon^2 L_i \gamma_i^2(\alpha) \kappa_i(\alpha)\eta,
\end{split}
\end{equation}
where $\kappa_i(\alpha)$ denotes the signed curvature of $\Gamma_i$ at $z_i(\alpha)$.
\end{lemma}

\begin{proof}
Since $z_i$ is the constant-speed parameterization of  $\Gamma_i$ (which has length $L_i$), we have $|z_i'|\equiv L_i$ and $\mathbf{n}(z_i(\alpha))=z_i'(\alpha)^\perp / L_i$. Taking the $\alpha$ and $\eta$ partial derivatives of \eqref{def_r_eps} directly yields \eqref{eq_partial}.

Putting the two partial derivatives into columns of a $2\times 2$ matrix and  computing the determinant, we have
\[
\begin{split}
\det (\nabla_{\alpha,\eta} R_i^\epsilon) &= \epsilon \gamma_i(\alpha) \frac{|z_i'(\alpha)|^2}{L_i} + \epsilon^2 \gamma_i^2(\alpha) \frac{z_i''(\alpha)^\perp \cdot z_i'(\alpha)}{L_i^2}\eta\\
& = \epsilon L_i \gamma_i(\alpha) - \epsilon^2 L_i \gamma_i^2(\alpha) \kappa_i(\alpha)\eta,
\end{split}
\]
where in the second equality we used that $z_i''(\alpha)=\kappa_i(\alpha)  \mathbf{n}(z_i(\alpha)) L_i^2$ (recall that $z_i$ has constant speed $L_i$). This finishes the proof.
\end{proof}
\begin{remark}\label{rmk_area}
We point out that for each $i=1,\dots,n+m$, the determinant formula \eqref{eq_det} immediately gives the following approximation of $|D^\epsilon_i|$, which will be helpful in the proofs later:
\begin{align}\label{P_nonzero}
\frac{|D^\epsilon_i|}{\epsilon} =\frac{1}{\epsilon} \int_{D^\epsilon_i}1dx =\frac{1}{\epsilon}\int_{S_i}\int_{-1}^{0}\det(\nabla_{\alpha,\eta} R^\epsilon_i(\alpha,\eta))\, d\eta d\alpha=  L_i\int_{S_i}\gamma_i(\alpha)d\alpha + O(\epsilon),
\end{align}
where the $O(\epsilon)$ error term has its absolute value bounded by $C\epsilon$, with $C$ only depending on $\|z_i\|_{C^2(S_i)}$ and $\|\gamma_i\|_{L^\infty(S_i)}$.
\end{remark}

Finally, let $D^\epsilon := \cup_{i=1}^{n+m} D_i^\epsilon$, and $\omega^\epsilon:\mathbb{R}^2\to \mathbb{R}$ is defined as 
\[
\omega^\epsilon(x) := \epsilon^{-1} 1_{D^\epsilon}(x) =\epsilon^{-1} \sum_{i=1}^{n+m} 1_{D_i^\epsilon}(x),
\]
and let 
\begin{equation}\label{def_v_eps}
\mathbf{v}^\epsilon = \nabla^\perp (\omega^\epsilon * \mathcal{N})
\end{equation} be the velocity field generated by $\omega^\epsilon$.

In the next lemma we aim to obtain some fine estimate of $\mathbf{v}^\epsilon$ in the thin vortex layer $D^\epsilon$. Our goal is to show that along each cross section of the thin layer (i.e. fix $i$ and $\alpha$, and let $\eta$ vary in $[-1,0]$), the function $\eta \mapsto \mathbf{v}^\epsilon(R_i^\epsilon(\alpha,\eta))$ is almost a linear function  in $\eta$, with the endpoint values (at $\eta=-1$ and $0$) being almost $\mathbf{v}^-(z_i(\alpha))$ and $\mathbf{v}^+(z_i(\alpha))$ respectively.

\begin{lemma}\label{lemma_v}
For $i=1,\dots,n+m$, assume $\Gamma_i$ and $\gamma_i$ satisfy \textbf{\textup{(H1)}}--\textbf{\textup{(H3)}}. Let
\[
g_i(\alpha, \eta) := BR(z_i(\alpha)) -  \Big(\eta+\frac{1}{2}\Big) [\mathbf{v}](z_i(\alpha)) \quad\text{ for }\alpha\in S_i,
\]
and note that $g_i(\alpha,0) = \mathbf{v}^{+}(z_i(\alpha))$ and $g_i(\alpha,-1) = \mathbf{v}^{-}(z_i(\alpha))$ (see Figure~\ref{fig_velocity} for an illustration of $g_i(\alpha,\eta)$).
Then for all sufficiently small $\epsilon>0$, for all $i = 1,\dots, n+m$ we have
\begin{equation}\label{diff_v}
|\mathbf{v}^\epsilon(R_i^\epsilon(\alpha,\eta)) - g_i(\alpha,\eta)| \leq C\epsilon^b |\log\epsilon| \quad\text{ for all }  \alpha\in S_i, \eta\in[-1,0],
\end{equation}
where $b\in(0,1)$ is as in \textup{\textbf{(H2)}}, and $C$ depends on $b$, $\max_i\|z_i\|_{C^2(S_i)}$, $\max_{i}\|\gamma_i\|_{C^b(S_i)}$,  $d_\Gamma$ and $F_\Gamma$.
\end{lemma}

\begin{figure}[h!]
\begin{center}
\includegraphics[scale=1.0]{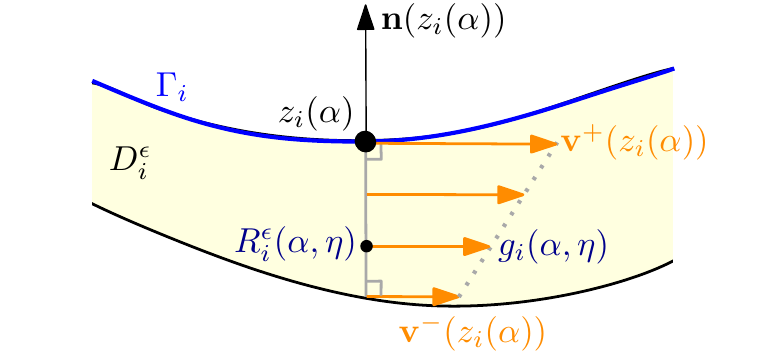}
\caption{Illustration of the definition of $g_i(\alpha,\cdot)$ (the orange arrows).\label{fig_velocity}}
\end{center}
\end{figure}

\begin{proof}
Let $i$ be any fixed index in $1,\dots, n+m$. We begin with breaking $\mathbf{v}^\epsilon$ into contributions from different components  $\{D_k^\epsilon\}_{k=1}^{n+m}$, namely
\[
\mathbf{v}^\epsilon(x) = \sum_{k=1}^{n+m}\mathbf{v}_k^\epsilon(x) := \sum_{k=1}^{n+m}\epsilon^{-1} \int_{D_i^\epsilon} K_2(x-y) dy,
\]
where the kernel $K_2$ is given by \eqref{def_k2}. Similarly, we can break $BR(z_i(\alpha))$ into 
$ \displaystyle 
BR(z_i(\alpha)) = \sum_{k=1}^{n+m} BR_k(z_i(\alpha)),$ where $BR_k$ is the contribution from the $k$-th integral in \eqref{def_BR}, and note that the PV symbol is only needed for $k= i$.

$\bullet$ \emph{Estimates for $k\neq i$ terms.} For any $k\neq i$, we aim to show that 
\begin{equation}\label{goal_ki}
|\mathbf{v}_k^\epsilon(R_i^\epsilon(\alpha,\eta)) - BR_k(z_i(\alpha))| \leq C\epsilon,
\end{equation}
where $C$ depends on $d_\Gamma, \max_k\|z_k\|_{C^2}$ and $\max_k\|\gamma_k\|_{L^\infty}$.
Applying a change of variable $y=R_k^\epsilon(\alpha',\eta')$, we can rewrite $\mathbf{v}_k^\epsilon$ as
\begin{equation}\label{v_int}
\begin{split}
\mathbf{v}_k^\epsilon(R_i^\epsilon(\alpha,\eta))&= \epsilon^{-1} \int_{D_k^\epsilon} K_2(R_i^\epsilon(\alpha,\eta)-y)\, dy \\
 &= \int_{S_k} \int_{-1}^0  \underbrace{K_2(R_i^\epsilon(\alpha,\eta) - R_k^\epsilon(\alpha',\eta'))}_{=:T_1}  \underbrace{\epsilon^{-1}\det(\nabla_{\alpha',\eta'} R_k^\epsilon(\alpha',\eta'))}_{=:T_2}  \, d\eta' d\alpha'.
\end{split}
\end{equation}
Using the facts that $R_i^\epsilon(\alpha,\eta) - R_k^\epsilon(\alpha',\eta') = z_i(\alpha)-z_k(\alpha') + O(\epsilon)$ as well as $|z_i(\alpha)-z_k(\alpha')|\geq d_{\Gamma}>0$ (recall that $d_{\Gamma}$ is as given in \eqref{def_d_gamma}), for all sufficiently small $\epsilon>0$ we have
$
T_1 = K_2(z_i(\alpha)-z_k(\alpha')) + O(\epsilon).
$
For $T_2$, the explicit formula \eqref{eq_det} for the determinant gives
$
T_2 =  L_k \gamma_k(\alpha') + O(\epsilon).
$
Plugging these into the above integral yields
\[
\mathbf{v}_k^\epsilon(R_i^\epsilon(\alpha,\eta)) = \int_{S_k}  K_2(z_i(\alpha)-z_k(\alpha')) L_k \gamma_k(\alpha')\, d\alpha' + O(\epsilon) = BR_k(z_i(\alpha))+O(\epsilon),
\]
finishing the proof of \eqref{goal_ki}.

$\bullet$ \emph{Estimates for the $k= i$ term.}  It will be more involved to control the $k=i$ term, and our goal is to show that
\begin{equation}\label{goal_ki2}
\left|\mathbf{v}_i^\epsilon(R_i^\epsilon(\alpha,\eta))   - BR_i(z_i(\alpha)) + \Big(\eta+\frac{1}{2}\Big) [\mathbf{v}](z_i(\alpha))  \right| \leq C\epsilon^b|\log\epsilon|.
\end{equation}
To begin with, we again rewrite $\mathbf{v}_i^\epsilon $ as in \eqref{v_int} with $k=i$, and plug in the formula \eqref{eq_det} for the determinant. This leads to 
\[
\begin{split}
\mathbf{v}_i^\epsilon(R_i^\epsilon(\alpha,\eta))&= \int_{S_k} \int_{-1}^0  K_2(R_i^\epsilon(\alpha,\eta) - R_i^\epsilon(\alpha',\eta')) 
\left(L_i \gamma_i(\alpha') - \epsilon L_i \gamma_i^2(\alpha') \kappa_i(\alpha')\eta'\right)
 \, d\eta' d\alpha'\\
 &=: I_1 + I_2,
 \end{split}
\]
where $I_1, I_2$ are the contributions from the two terms in the last parenthesis respectively. 
Let us control $I_2$ first, and we claim that 
\begin{equation}\label{I2}
|I_2| \leq C \epsilon|\log\epsilon|.
\end{equation} 
Using \eqref{r_diff} of Lemma \ref{lemma_injection} and the fact that $|K_2(x)|\leq |x|^{-1}$, we can bound $I_2$ as 
\begin{equation}\label{estimate_I2}
\begin{split}
|I_2| &= \left| \int_{S_k} \int_{-1}^0  K_2(R_i^\epsilon(\alpha,\eta) - R_i^\epsilon(\alpha',\eta')) 
\, \epsilon L_i \gamma_i^2(\alpha') \kappa_i(\alpha') \eta'
 \, d\eta' d\alpha' \right|
\\&\leq  C \epsilon \int_{S_k} \int_{-1}^0 \frac{\gamma_i(\alpha')}{|\alpha'-\alpha| + \epsilon |\gamma_i(\alpha')\eta'-\gamma_i(\alpha)\eta|}  \, d\eta' d\alpha'\\
&\leq C \epsilon \int_{S_k} \int_{-\|\gamma_i\|_{\infty}}^{\|\gamma_i\|_{\infty}} \frac{1}{|\alpha'-\alpha| + \epsilon |\theta'| }\, d\theta' d\alpha' \quad(\theta':=\gamma_i(\alpha')\eta'-\gamma_i(\alpha)\eta)\\
&\leq C\epsilon \int_{-1/\epsilon}^{1/\epsilon} \int_{-\|\gamma_i\|_{\infty}}^{\|\gamma_i\|_{\infty}} \frac{1}{|\beta'| +  |\theta'| }\, d\theta' d\beta' \quad(\beta':=\epsilon^{-1}(\alpha'-\alpha))\\
&\leq C\epsilon|\log \epsilon|
\end{split}
\end{equation}
where $C$ depends on $\|z_i\|_{C^2}$ and $\|\gamma_i\|_{L^\infty}$. 

In the rest of the proof we focus on estimating 
$
I_1 =  \displaystyle\int_{S_k} \int_{-1}^0  K_2(R_i^\epsilon(\alpha,\eta) - R_i^\epsilon(\alpha',\eta')) 
L_i \gamma_i(\alpha')  \, d\eta' d\alpha'.
$
For $t \in [0,1]$, let us define
\begin{align}
&f(\alpha,\alpha',\eta,\eta';t) := R_i^\epsilon(\alpha,\eta-t\eta') - R_i^\epsilon(\alpha',\eta'-t\eta'), \nonumber \\
&J(t) :=  \int_{S_k} \int_{-1}^0  K_2(f(\alpha,\alpha',\eta,\eta';t)) 
L_i \gamma_i(\alpha')  \, d\eta' d\alpha'.\label{def_J}
\end{align}
Note that in the definition of $f$, the argument $\eta-t\eta'$ of $R_i^\epsilon$ belongs to $[-1,1]$, instead of $[-1,0]$ as in the original definition of \eqref{def_r_eps}. Here $R_i^\epsilon(\alpha,\eta-t\eta')$ is defined as in the formula \eqref{def_r_eps}, even though it might not belong to $D^\epsilon_i$.
Clearly, $J(0) = I_1$. The motivation for us to define such $f$ and $J(t)$ is that at $t=1$, we have
\begin{equation}\label{J1}
\begin{split}
J(1) &=  \int_{S_k} \int_{-1}^0  K_2(R_i^\epsilon(\alpha,\eta-\eta') -z_i(\alpha')) 
L_i \gamma_i(\alpha')  \, d\eta' d\alpha'=  \int_{-1}^0 \mathbf{v}_i(R_i^\epsilon(\alpha,\eta-\eta'))  \, d\eta',
\end{split}
\end{equation}
where $\mathbf{v}_i$ is the velocity field generated by the sheet $\Gamma_i$.  
Recall that $\mathbf{v}_i$ has a jump across $\Gamma_i$, where we denote its limits on two sides by $\mathbf{v}_i^\pm$. Using Lemma~\ref{lemma_v_jump}, which we will prove momentarily, we have
\begin{equation}\label{v_approx}
\mathbf{v}_i(R_i^\epsilon(\alpha,\eta-\eta')) = \begin{cases} \mathbf{v}_i^+(z_i(\alpha)) + O(\epsilon^b |\log\epsilon|) &\text{ if } \eta-\eta' \in (0,2),\\
\mathbf{v}_i^-(z_i(\alpha))+ O(\epsilon^b |\log\epsilon|)& \text{ if }  \eta-\eta' \in (-2,0).
\end{cases}
\end{equation}

We can then split the integration domain on the right hand side of \eqref{J1} into $\eta'\in(-1,\eta)$ and $\eta'\in(\eta,0)$, and use \eqref{v_approx} to approximate the integrand in each interval. This gives 
\begin{equation}\label{J1_2}
\begin{split}
J(1)  &=  (\eta+1) \mathbf{v}_i^+(z_i(\alpha)) - \eta \mathbf{v}_i^-(z_i(\alpha)) + O(\epsilon^b |\log\epsilon|)\\
&= BR_i(z_i(\alpha)) -\Big(\eta+\frac{1}{2}\Big) [\mathbf{v}](z_i(\alpha))+ O(\epsilon^b |\log\epsilon|),
\end{split}
\end{equation}
where in the last step we used that $[\mathbf{v}](z_i(\alpha)) = [\mathbf{v}_i](z_i(\alpha))$, since all other $\mathbf{v}_k$ with $k\neq i$ are continuous across $\Gamma_i$.

Finally, it remains to control $|J(0)-J(1)|$. Note that by \eqref{r_diff}, we have 
\[
f(\alpha,\alpha',\eta,\eta';t) \geq c_0\big(|\alpha-\alpha'| + \epsilon  |\gamma_i(\alpha')\eta'-\gamma_i(\alpha)\eta|\big).
\] 
In addition, we have
\[
\left|\frac{\partial}{\partial t} f(\alpha,\alpha',\eta,\eta';t)\right| = \left|\epsilon \big(\gamma_i(\alpha) \mathbf{n}(z_i(\alpha))-\gamma_i(\alpha') \mathbf{n}(z_i(\alpha'))\big)\eta'\right| \leq C\epsilon |\alpha-\alpha'|^b,
\]
where the last inequality follows from \textbf{(H2)} and the fact that  $\mathbf{n}(z_i(\alpha)) \in C^1(S_i)$.  Therefore, for any $t\in (0,1)$, taking the $t$ derivative of \eqref{def_J} and using that $|\nabla K_2(x)| \leq |x|^{-2}$, we have
\[
\begin{split}
|J'(t)| &\leq C \int_{S_k} \int_{-1}^0 \frac{ \epsilon |\alpha-\alpha'|^b \gamma_i(\alpha') }{\big(|\alpha-\alpha'| + \epsilon  |\gamma_i(\alpha')\eta'-\gamma_i(\alpha)\eta|\big)^{2}}d\eta' d\alpha'\\
&\leq C \epsilon \int_{S_k} \int_{-1}^0 \frac{ \gamma_i(\alpha') }{|\alpha-\alpha'|^{1-b}\left(|\alpha-\alpha'| + \epsilon |\gamma_i(\alpha')\eta'-\gamma_i(\alpha)\eta|\right)}d\eta' d\alpha' \\
& \leq C \epsilon^b \int_{-1/\epsilon}^{1/\epsilon} \int_{-\|\gamma_i\|_{\infty}}^{\|\gamma_i\|_{\infty}} \frac{1}{|\beta'|^{1-b}(|\beta'| + |\theta'|) }\, d\theta' d\beta' \quad(\theta':=\gamma_i(\alpha')\eta'-\gamma_i(\alpha)\eta,\, \beta':=\epsilon^{-1}(\alpha'-\alpha))\\
&\leq  C \epsilon^b  \int_{-1/\epsilon}^{1/\epsilon} |\beta'|^{b-1}\log\left(1+\frac{\|\gamma_i\|_{L^\infty}}{|\beta'|}\right) d\beta'\\
& \leq C\epsilon^b,
\end{split}
\]
where $C$ depends on $b$, $\|\gamma_i\|_{C^b(S_i)}$, $\|z_i\|_{C^2(S_i)}$ and $F_\Gamma$. This leads to \[
|J(1)-I_1| = |J(1)-J(0)| \leq C\epsilon^b |\log\epsilon|.
\]
Finally, combining this with \eqref{J1_2} and \eqref{I2} yields \eqref{goal_ki2}, finishing the proof of the $k=i$ case.
We can then conclude the proof by taking the sum of this estimate with all the $k\neq i$ estimates in \eqref{goal_ki}. 
\end{proof}

The following lemma proves \eqref{v_approx}. Let $\mathbf{v}_i$ be the velocity field generated by the sheet $\Gamma_i$, which is smooth in $\mathbb{R}^2\setminus \Gamma_i$, and has a discontinuity across $\Gamma_i$. It is known that $\mathbf{v}_i$ converges to $\mathbf{v}_i^\pm$ respectively on the two sides of $\Gamma_i$ \cite{Majda-Bertozzi:vorticity-incompressible-flow}.  However, we were unable to find a quantitative convergence rate (in terms of the distance from the point to $\Gamma_i$) in the literature, especially under the assumption that $\gamma_i$ is only in $C^b(S_i)$ for the open curves. Below we prove such an estimate.

\begin{lemma}\label{lemma_v_jump}
For $i=1,\ldots,n+m$, let $\mathbf{v}_i$ be the velocity field generated by the sheet $\Gamma_i$, given by
\[
\mathbf{v}_i(x) := \int_{S_i}  K_2(x - z_i(\alpha'))\, \gamma_i(\alpha') |z_i'(\alpha')| \,d\alpha' \quad\text{ for }x\in\mathbb{R}^2\setminus\Gamma_i.
\]
Then there exist constants $C,\epsilon_0 >0$ depending on on $b$ (as in \textup{\textbf{(H2)}}), $\|z_i\|_{C^2(S_i)}$, $\|\gamma_i\|_{C^b(S_i)}$ and $F_\Gamma$, such that for all $\epsilon \in (0,\epsilon_0)$ and $\eta\in(-2,2)$ we have
\begin{align}
&\left| \mathbf{v}_i(R^\epsilon_i(\alpha,\eta)) - \mathbf{v}_i^+(z_i(\alpha)) \right| \le C \epsilon^b|\log \epsilon | \quad \text{ if }\eta \in (0,2),\label{positive_eta}\\
&\left| \mathbf{v}_i(R^\epsilon_i(\alpha,\eta)) - \mathbf{v}_i^-(z_i(\alpha)) \right| \le C \epsilon^b|\log\epsilon| \quad \text{ if }\eta \in (-2,0), \label{negative_eta}
\end{align}
where
\[\mathbf{v}_i^{+} = BR_i(z_i(\alpha)) + \frac{\mathbf{n}(z_i(\alpha))^\perp\gamma_i(\alpha)}{2},\quad
\mathbf{v}_i^{-} = BR_i(z_i(\alpha)) - \frac{\mathbf{n}(z_i(\alpha))^\perp\gamma_i(\alpha)}{2},
\]
and $BR_i$ is the contribution from the $i$-th integral in \eqref{def_BR}.
\end{lemma}
\begin{proof} 
We will show \eqref{positive_eta} only since \eqref{negative_eta} can be treated in the same way.   From the definition of $R^\epsilon_i$ in \eqref{def_r_eps}, we have
\begin{align*}
\mathbf{v}_i(R^\epsilon_i(\alpha,\eta)) 
& = \frac{L_i}{2\pi}\int_{S_i}\frac{\left(z_i(\alpha)-z_i(\alpha') \right)^{\perp} \gamma_i(\alpha')}{\left| z_i(\alpha)-z_i(\alpha')+\epsilon\eta\mathbf{n}(z_i(\alpha))\gamma_i(\alpha)\right|^2} \,d\alpha'  + \frac{L_i}{2\pi}\int_{S_i}\frac{\epsilon\eta\mathbf{n}(z_i(\alpha))^{\perp}\gamma_i(\alpha)  \gamma_i(\alpha')}{\left| z_i(\alpha)-z_i(\alpha')+\epsilon\eta\mathbf{n}(z_i(\alpha))\gamma_i(\alpha)\right|^2}\, d\alpha'\\
& =: A_1 + A_2.
\end{align*}
We claim that for all $\epsilon>0$ sufficiently small and $\eta\in[0,2)$, we have
\begin{align}
&|A_1 - BR_i(z(\alpha)) | \le C\epsilon^b|\log\epsilon|, \label{tangential_part}\\
&\left| A_2 - \frac{\mathbf{n}(z(\alpha))^\perp\gamma(\alpha)}{2} \right| \le C\epsilon^b, \label{normal_part}
\end{align}
and note that these two claims immediately yield \eqref{positive_eta}. From now on, let us fix $i\in \left\{1,\ldots,n+m\right\}$ and omit it in the notation for simplicity. Throughout this proof, let us denote 
\[
\mathbf{y}(\alpha,\alpha'):= z(\alpha)-z(\alpha')
\quad\text{ and }\quad
\mathbf{c}(\alpha):=\epsilon\eta\mathbf{n}(z(\alpha))\gamma(\alpha),\]
 so that 
\[
A_1 = \frac{L}{2\pi} \int_{S} \frac{\mathbf{y}^\perp(\alpha,\alpha') \gamma(\alpha')}{|\mathbf{y}(\alpha,\alpha')+\mathbf{c}(\alpha)|^2} d\alpha', \quad A_2 = \frac{L}{2\pi} \int_S  \frac{\mathbf{c}^\perp(\alpha)\gamma(\alpha')}{|\mathbf{y}(\alpha,\alpha')+\mathbf{c}(\alpha)|^2 } d\alpha'.
\]
Note that 
\begin{equation}\label{y222}
F_\Gamma^{-1} |\alpha-\alpha'| \leq |\mathbf{y}(\alpha,\alpha')|\leq \|z\|_{C^1}|\alpha-\alpha'|.
\end{equation}
For the closed curves with $i=1,\dots,n$, since $z$ has period 1, we can always set $\alpha-\alpha' \in [-\frac{1}{2},\frac{1}{2})$ in this proof.

Applying \eqref{r_diff} (with $\eta'=0$), we have
\begin{equation}\label{y111}
|\mathbf{y}(\alpha,\alpha')+\mathbf{c}(\alpha)|^2 \geq c_0(|\alpha-\alpha'|^2 + \epsilon^2\eta^2 \gamma^2(\alpha)) = c_0(|\alpha-\alpha'|^2 +|\mathbf{c}(\alpha)|^2).
\end{equation}
Since $z'(\alpha)=L\mathbf{s}(z(\alpha))$, let us define
\[
\tilde{\mathbf{y}}(\alpha,\alpha') := L\mathbf{s}(z(\alpha))(\alpha-\alpha'),
\] which is a close approximation of $\mathbf{y}$ in the sense that 
\begin{equation}\label{diff_y}
|\mathbf{y}(\alpha,\alpha')-\tilde{\mathbf{y}}(\alpha,\alpha')|\leq \|z\|_{C^2}(\alpha-\alpha')^2.
\end{equation} 
Using $\mathbf{s}(z(\alpha)) \perp \mathbf{n}(z(\alpha))$, we have
\begin{equation}\label{eq_tilde_y}
|\tilde{\mathbf{y}}(\alpha,\alpha')+\mathbf{c}(\alpha)|^2 = L^2|\alpha-\alpha'|^2 + \epsilon^2\eta^2 \gamma^2(\alpha) = L^2 |\alpha-\alpha'|^2 + |\mathbf{c}(\alpha)|^2.
\end{equation}
From now on, for notational simplicity, we compress the dependence of $\mathbf{y}(\alpha,\alpha'), \tilde{\mathbf{y}}(\alpha,\alpha'), \mathbf{c}(\alpha)$ on $\alpha$ and $\alpha'$ in the rest of the proof.

$\bullet$ \emph{Estimate \eqref{tangential_part}.} Note that $BR_i(z(\alpha))$ can also be written using the above notations as
\[
BR_i(z(\alpha)) = \frac{L}{2\pi} PV\int_S \frac{\mathbf{y}^\perp}{|\mathbf{y}|^2} \gamma(\alpha')d\alpha,
\]
thus $A_1-BR_i(z(\alpha))$ can be written as follows:
\[
\begin{split}
A_1-BR_i(z(\alpha)) &= \frac{L}{2\pi} PV \int_S \underbrace{\left( \frac{\mathbf{y}^\perp}{|\mathbf{y}+\mathbf{c}|^2}-\frac{\mathbf{y}^\perp}{|\mathbf{y}|^2}\right)}_{=:\mathbf{f}(\mathbf{y},\mathbf{c})}\gamma(\alpha') d\alpha'\\
&=  \frac{L}{2\pi}  \int_S \mathbf{f}(\mathbf{y},\mathbf{c}) (\gamma(\alpha')-\gamma(\alpha)) d\alpha' +  \frac{L\gamma(\alpha)}{2\pi} PV \int_S \mathbf{f}(\mathbf{y},\mathbf{c}) d\alpha' \\ 
&=: A_{11}+A_{12}.
\end{split}
\]
A direct computation gives
\begin{equation}\label{def_f}
\mathbf{f}(\mathbf{y},\mathbf{c})
= -\frac{\mathbf{y}^\perp}{|\mathbf{y}|^2}\,\frac{2\mathbf{y}\cdot\mathbf{c}+|\mathbf{c}|^2}{|\mathbf{y}+\mathbf{c}|^2}.
\end{equation}
Since $\mathbf{y}\cdot\mathbf{c}=(\mathbf{y}-\tilde{\mathbf{y}})\cdot\mathbf{c} \leq C|\alpha-\alpha'|^2 |\mathbf{c}|,$ (where we use $\tilde{\mathbf{y}}\perp\mathbf{n}(z(\alpha))$ and \eqref{diff_y}), combining this with \eqref{y222} and \eqref{y111} gives a crude bound
\[
|\mathbf{f}(\mathbf{y},\mathbf{c})| \lesssim \frac{|\alpha-\alpha'|^2 |\mathbf{c}| + |\mathbf{c}|^2}{|\alpha-\alpha'| (|\alpha-\alpha'|^2+|\mathbf{c}|^2)}.
\]
Plugging this into $A_{11}$ and using the H\"older continuity of $\gamma$, we have
\[
\begin{split}
|A_{11}| &\lesssim \int_{S}\frac{|\alpha-\alpha'|^2 |\mathbf{c}| + |\mathbf{c}|^2}{|\alpha-\alpha'| (|\alpha-\alpha'|^2+|\mathbf{c}|^2)} |\alpha-\alpha'|^b d\alpha'\\
&\lesssim  \int_{|\theta|<|\mathbf{c}|}(|\theta|^{1+b}|\mathbf{c}|^{-1}+|\theta|^{b-1}) d\theta +  \int_{|\mathbf{c}|\leq |\theta|\leq 1} (|\mathbf{c}| |\theta|^{b-1} + |\mathbf{c}|^2|\theta|^{b-3})d\theta  \quad(\theta:=\alpha'-\alpha)\\
&\lesssim |\mathbf{c}|^b \leq C\epsilon^b,
\end{split}
\]
where the last step follows from the fact that $|\mathbf{c}|\leq 2\epsilon\|\gamma\|_\infty$.
Now let us turn to $A_{12}$, which requires a more delicate estimate of $\mathbf{f}(\mathbf{y},\mathbf{c})$. Let us break $A_{12}$ as
\[
\begin{split}
A_{12} &= \frac{L\gamma(\alpha)}{2\pi} \int_S (\mathbf{f}(\mathbf{y},\mathbf{c})-\mathbf{f}(\tilde{\mathbf{y}},\mathbf{c})) d\alpha' +  \frac{L\gamma(\alpha)}{2\pi}PV \int_S \mathbf{f}(\tilde{\mathbf{y}},\mathbf{c}) d\alpha'=: B_1 + B_2.
\end{split}
\]
For $B_1$, let us take the gradient of $\mathbf{f}(\mathbf{y},\mathbf{c})$ (as in \eqref{def_f}) in the first variable. An elementary computation yields that
\begin{equation}\label{grad_f}
|\nabla_{\mathbf{x}} \mathbf{f}(\mathbf{x},\mathbf{c})| \leq C|\mathbf{x}|^{-2}\min\Big\{1, \frac{|\mathbf{c}|}{|\mathbf{x}|}\Big\}
\end{equation}
as long as $\mathbf{x}$ satisfies
\begin{equation}\label{cond_x}
|\mathbf{x} +  \mathbf{c}|^2 \geq c_0(|\mathbf{x}|^2+|\mathbf{c}|^2).
 \end{equation}
We point out that $\mathbf{x}=\xi\mathbf{y}+(1-\xi)\tilde{\mathbf{y}}$ indeed satisfies \eqref{cond_x} for all $\xi\in[0,1]$: to see this, in the proof of Lemma~\ref{lemma_injection}, if we replace $T_1$ in \eqref{temp001} by  $\xi\mathbf{y}+(1-\xi)\tilde{\mathbf{y}}$, one can easily check the proof still goes through for $\xi\in[0,1]$. In addition, for any $\xi\in[0,1]$ we also have 
\begin{equation}\label{temp000}
|\xi\mathbf{y}+(1-\xi)\tilde{\mathbf{y}}| \geq c_0 |\alpha-\alpha'|.
\end{equation}
Thus the gradient estimate \eqref{grad_f} together with \eqref{diff_y} and \eqref{temp000} yields
\[
|f(\mathbf{y}, \mathbf{c}) -f(\tilde{\mathbf{y}}, \mathbf{c})| \lesssim \min\{1, |\mathbf{c}||\alpha-\alpha'|^{-1}\}\lesssim \min\{1,\epsilon|\alpha-\alpha'|^{-1}\},
\]
and plugging this into $B_1$ gives
\[
\begin{split}
|B_1| &\lesssim \epsilon + \int_{\epsilon<|\alpha-\alpha'|<1} \epsilon|\alpha-\alpha'|^{-1} d\alpha' \lesssim \epsilon|\log\epsilon|.
\end{split}
\] 
As for $B_2$, using the definition of $\tilde{\mathbf{y}}$, the identity \eqref{eq_tilde_y} and the fact that $\tilde{\mathbf{y}}\cdot\mathbf{c}=0$, we have
\[
\begin{split}
B_2 &=  \frac{L\gamma(\alpha)}{2\pi} PV \int_S -\frac{\tilde{\mathbf{y}}^\perp}{|\tilde{\mathbf{y}}|^2}\,\frac{|\mathbf{c}|^2}{|\tilde{\mathbf{y}}+\mathbf{c}|^2} d\alpha'\\
&=  \frac{L\gamma(\alpha)|\mathbf{c}|^2 \mathbf{n}(z(\alpha))}{2\pi L } PV \int_S \frac{\alpha'-\alpha}{|\alpha'-\alpha|^2 (L^2 |\alpha'-\alpha|^2 + |\mathbf{c}|^2)}d\alpha'.
\end{split}
\]
For the closed curves $i=1,\dots,n$, we immediately have $B_2=0$ since $\alpha-\alpha' \in [-\frac{1}{2},\frac{1}{2})$, and the integrand is an odd function of $\alpha'-\alpha$. 

For the open curves $i=n+1,\dots,n+m$, the above integral becomes
\[
\begin{split}
B_2 &= \frac{L\gamma(\alpha)|\mathbf{c}|^2 \mathbf{n}(z(\alpha))}{2\pi L } PV \int_{-\alpha}^{1-\alpha} \frac{\theta}{|\theta|^2 (L^2 |\theta|^2 + |\mathbf{c}|^2)}d\theta \quad(\theta:=\alpha'-\alpha)\\
&= \frac{L\gamma(\alpha)|\mathbf{c}|^2 \mathbf{n}(z(\alpha))}{2\pi L } \int_{\alpha}^{1-\alpha} \frac{\theta}{\theta^2 (L^2 \theta^2 + |\mathbf{c}|^2)}d\theta,
\end{split}
\]
where in the second inequality we used that the integral in $[-\alpha,\alpha]$ gives zero contribution to the principal value, since the integrand is odd. 

Next we discuss two cases. If $\alpha > |\mathbf{c}|$, we bound the integrand by $C\theta^{-3}$, which gives 
\[
|B_2| \leq C \gamma(\alpha) |\mathbf{c}|^2 \alpha^{-2} \leq C |\mathbf{c}|^2 \alpha^{b-2} \leq  C |\mathbf{c}|^b \leq C\epsilon^b.
\] where the second inequality follows from the assumption $\gamma(0)=0$ for an open curve in \textbf{(H3)}, as well as the H\"older continuity of $\gamma$. And if $0<\alpha\leq |\mathbf{c}|$, the integrand can be bounded above by $\theta^{-1}|\mathbf{c}|^{-2}$, which immediately leads to 
\[
|B_2| \leq C\gamma(\alpha)|\log\alpha| \leq C|\mathbf{c}|^b|\log|\mathbf{c}|| \leq C\epsilon^b|\log\epsilon|.
\] In both cases we have
$
|B_2| \leq C\epsilon^b|\log\epsilon|,
$
and combining it with the $B_1$ and $A_{11}$ estimates gives \eqref{tangential_part}.

$\bullet$ \emph{Estimate \eqref{normal_part}.} We break $A_2$ into
\begin{align*}
A_2 & = \frac{L\mathbf{c}^\perp}{2\pi} \int_S  \frac{\gamma(\alpha')-\gamma(\alpha)}{|\mathbf{y}+\mathbf{c}|^2} d\alpha' +  \frac{L\mathbf{c}^\perp\gamma(\alpha)}{2\pi} \int_S \left( \frac{1}{|\mathbf{y}+\mathbf{c}|^2} - \frac{1}{|\tilde{\mathbf{y}}+\mathbf{c}|^2} \right)d\alpha' +\frac{L\mathbf{c}^\perp\gamma(\alpha)}{2\pi} \int_S \frac{1}{|\tilde{\mathbf{y}}+\mathbf{c}|^2} d\alpha'\\
&=: A_{21} + A_{22} + A_{23}.
\end{align*}
For $A_{21}$, \eqref{y111} and the H\"older continuity of $\gamma$ immediately lead to 
\begin{equation}\label{a21}
|A_{21}| \leq C|\mathbf{c}|\int_{S} \frac{|\alpha-\alpha'|^b}{|\alpha-\alpha'|^2+|\mathbf{c}|^2} d\alpha' \leq |\mathbf{c}|^b \leq C\epsilon^b.
\end{equation}
For $A_{22}$, its integrand can be controlled as
\[
 \left|\frac{1}{|\mathbf{y}+\mathbf{c}|^2} - \frac{1}{|\tilde{\mathbf{y}}+\mathbf{c}|^2}\right|\leq \frac{|\mathbf{y}-\mathbf{\tilde y}| (|\mathbf{y}+\mathbf{c}|+|\tilde{\mathbf{y}}+\mathbf{c}|)}{|\mathbf{y}+\mathbf{c}|^2 |\tilde{\mathbf{y}}+\mathbf{c}|^2} \leq \frac{C|\alpha-\alpha'|^2}{(|\alpha-\alpha'|^2+|\mathbf{c}|^2)^{3/2}},
\]
where the last step follows from \eqref{y111}, \eqref{diff_y} and \eqref{eq_tilde_y}. This allows us to control $A_{22}$ as
\begin{equation}\label{a22}
|A_{22}| \leq C|\mathbf{c}| \int_{-1}^1 \frac{\theta^2}{(\theta^2 + |\mathbf{c}|^2)^{3/2}}d\theta \leq C|\mathbf{c}| \,\big|\log|\mathbf{c}|\big| \leq C\epsilon|\log\epsilon|.
\end{equation}
Finally, for the $A_{23}$ term, \eqref{eq_tilde_y} gives
\[
A_{23} = \frac{L\mathbf{c}^\perp\gamma(\alpha)}{2\pi} \int_S \frac{1}{L^2|\alpha'-\alpha|^2+|\mathbf{c}|^2} d\alpha' = \frac{\mathbf{n}^\perp(\alpha)\gamma(\alpha)}{2\pi } \int_I \frac{1}{\theta^2 + 1} d\theta \quad(\text{set }\theta := \frac{L(\alpha'-\alpha)}{|\mathbf{c}|}),
\]
where the integration interval $I=(-\frac{L}{2|\mathbf{c}|},\frac{L}{2|\mathbf{c}|})$ for $i=1,\dots,n$, and $I=(-\frac{L\alpha}{|\mathbf{c}|},\frac{L(1-\alpha)}{|\mathbf{c}|})$ for $i=n+1,\dots,n+m$, and in the last equality we also used that $\frac{\mathbf{c}^\perp}{|\mathbf{c}|}=\mathbf{n}^\perp$. For $i=1,\dots,n$, one can easily check that
\[
\left|\int_I \frac{1}{\theta^2 + 1} d\theta-\pi\right| = 2\int_{\frac{L}{2|\mathbf{c}|}}^\infty \frac{1}{\theta^2+1} d\theta \leq C|\mathbf{c}|\leq C\epsilon,
\]
which immediately leads to
\[
\left|A_{23}-\frac{\mathbf{n}(z(\alpha))^\perp\gamma(\alpha)}{2}\right| =\left|\frac{\mathbf{n}^\perp(\alpha)\gamma(\alpha)}{2\pi } \left(\int_I \frac{1}{\theta^2 + 1} d\theta-\pi\right) \right| \leq C\epsilon
\]
for $i=1,\dots,n$. Next we turn to the open curves $i=n+1,\dots,n+m$, and let us assume $\alpha\in[0,\frac{1}{2}]$ without loss of generality. In this case we have
\[
\left|\int_I \frac{1}{\theta^2 + 1} d\theta-\pi\right| = \int^{-\frac{L\alpha}{|\mathbf{c}|}}_{-\infty}\frac{1}{\theta^2+1} d\theta +\int_{\frac{L(1-\alpha)}{|\mathbf{c}|}}^{\infty}\frac{1}{\theta^2+1} d\theta \leq \min\left\{C\frac{|\mathbf{c}|}{\alpha},\frac{\pi}{2}\right\}+C\epsilon.
\]
where we used $1-\alpha>\frac{1}{2}$ to control the second integral by $C\epsilon$. Using the above inequality as well as the fact that $\gamma(\alpha)\leq C\alpha^b$ due to \textbf{(H3)}, we have
\[
\left|A_{23}-\frac{\mathbf{n}(z(\alpha))^\perp\gamma(\alpha)}{2}\right|=\frac{\gamma(\alpha)}{2\pi } \left|\int_I \frac{1}{\theta^2 + 1} d\theta-\pi\right| \leq C\alpha^b \min\left\{\frac{|\mathbf{c}|}{\alpha},1\right\}+C\epsilon \leq C(|\mathbf{c}|^b+\epsilon)\leq C\epsilon^b
\]
for $i=n+1,\dots,n+m$.
Finally, combining the $A_{23}$ estimates together with \eqref{a21} and \eqref{a22} yields \eqref{normal_part}.
\end{proof}

\section{Constructing a divergence-free perturbation}
\label{subsec_p}

In this section, we aim to construct a divergence-free velocity field $\mathbf{u}^\epsilon:D^\epsilon\to\mathbb{R}^2$, such that $-\mathbf{u}^\epsilon$ tends to make each $D_i^\epsilon$ ``more symmetric''.
Let $\mathbf{u}^\epsilon: D^\epsilon\to\mathbb{R}^2$ be given by
\begin{equation}\label{u_def}
\mathbf{u}^\epsilon := x + \nabla p^\epsilon\quad\text{ in } D^\epsilon,
\end{equation}
where the function $p^\epsilon: \overline{D^\epsilon} \to \mathbb{R}$ is chosen such that 
\begin{equation}\label{u_div_free}
\nabla\cdot \mathbf{u}^\epsilon = 0\quad\text{ in }D^\epsilon,
\end{equation}
 and on each connected component $l$ of $\partial D^\epsilon$, $u^\epsilon$ satisfies
\begin{equation}\label{int_boundary}
\int_l \mathbf{u}^\epsilon \cdot n \,d\sigma = 0,
\end{equation}
where $n$ is the unit normal of $l$ pointing outwards of $D^\epsilon$. Note that $\partial D^\epsilon$ has a total of $2n+m$ connected components:  $D_i^\epsilon$ is doubly-connected for $i=1,\dots,n$ (denote its outer and inner boundaries by $\partial D_{i,\text{out}}^\epsilon$ and $\partial D_{i,\text{in}}^\epsilon$; note that $\partial D_{i,\text{in}}^\epsilon$ coincides with $\Gamma_i$), whereas it is simply-connected for $i=n+1,\dots,n+m$ (denote its boundary by $\partial D_i^\epsilon$).

Next we show that there indeed exists a function $p^\epsilon$ so that $\mathbf{u}^\epsilon$ satisfies \eqref{u_div_free}--\eqref{int_boundary}. Clearly, \eqref{u_div_free} requires that $p^\epsilon$ satisfies 
\begin{equation}\label{def_p_laplacian}
\Delta p^\epsilon = -2 \quad\text{ in }D^\epsilon.
\end{equation}
As for the boundary conditions, we let 
\begin{equation}\label{bdry_curve}
p^\epsilon|_{\partial D_i^\epsilon} = 0 \quad\text{for }i=n+1,\dots,n+m,
\end{equation} so the  divergence theorem yields that \eqref{int_boundary} is satisfied for each $l = \partial D_i^\epsilon$ for $i=n+1,\dots,n+m$. As for $i=1,\dots,n$, we define 
\begin{equation}\label{bdry_loop}
p^\epsilon = \begin{cases}
0 & \text{ on } \partial D_{i,\text{out}}^\epsilon \\
c_i^\epsilon  &\text { on } \partial D_{i,\text{in}}^\epsilon = \Gamma_i
\end{cases}\quad\text{ for } i = 1,\dots, n,
\end{equation} where $c_i^\epsilon>0$ is the unique constant such that 
\begin{equation}\label{bdry_loop2}
\int_{\partial U_i} \nabla p^\epsilon \cdot n d\sigma = -2|U_i| \quad\text{ for } i = 1,\dots, n,
\end{equation} where $U_i$ is the domain enclosed by $\partial D_{i,\text{in}}^\epsilon=\Gamma_i$ (thus $U_i$ is independent of $\epsilon$), and $n$ is the outer normal of $U_i$ (thus the inner normal of $D_i^\epsilon$). The existence of $c_i^\epsilon$ is guaranteed by \cite[Lemma 2.5]{GomezSerrano-Park-Shi-Yao:radial-symmetry-stationary-solutions}. One can then check that $\int_{\partial U_i} \mathbf{u}^\epsilon \cdot n d\sigma = 0$. Applying the divergence theorem in $D^\epsilon_i$ then gives us that $\int_{\partial D_{i,\text{out}}^\epsilon} \mathbf{u}^\epsilon \cdot n d\sigma = 0$ as well.

In \cite{GomezSerrano-Park-Shi-Yao:radial-symmetry-stationary-solutions} we proved a rearrangement inequality for such $p^\epsilon$ in a similar spirit of Talenti's  rearrangement inequality for elliptic equations \cite{Talenti:rearrangements}, which we state below.

\begin{lemma}[{\cite[Proposition 2.6]{GomezSerrano-Park-Shi-Yao:radial-symmetry-stationary-solutions}}]
The function $p^\epsilon: \overline{D^\epsilon}\to \mathbb{R}$ defined in \eqref{def_p_laplacian}--\eqref{bdry_loop2} satisfies the following in each $D_i^\epsilon$ for $i=1,\dots,n+m$:
\begin{equation}\label{ineq_sup}
\sup_{D_i^\epsilon} p^\epsilon \leq \frac{|D_i^\epsilon|}{2\pi},
\end{equation}
and
\begin{equation}\label{ineq_int}
\int_{D_i^\epsilon} p^\epsilon(x) dx \leq \frac{|D_i^\epsilon|^2}{4\pi}.
\end{equation}
Moreover, each inequality above achieves equality if and only $D_i^\epsilon$ is either a disk or an annulus.
\end{lemma}

Note that the inequalities \eqref{ineq_sup}--\eqref{ineq_int} hold for any domain with $C^{1,\alpha}$ boundary. Even though the inequalities are strict when $D_i^\epsilon$ is non-radial, they are not strong enough to rule out non-radial vortex sheets, as we need quantitative versions of strict inequalities that are still valid in the $\epsilon\to  0^+$ limit. As we will see in the proof of Proposition~\ref{prop2_i}, the key step is to show that if some $\Gamma_i$ is either not a circle or does not have a constant $\gamma_i$, then the following quantitative version of \eqref{ineq_int} holds: $\epsilon^{-2}\left(\frac{|D_i^\epsilon|^2}{4\pi}-\int_{D_i^\epsilon} p^\epsilon(x) dx\right)\geq c_0>0$, where $c_0$ is independent of $\epsilon$.

In order to upgrade \eqref{ineq_int} into a quantitative version, we need to obtain some fine estimates for $p^\epsilon$ that take into account the shape of the thin domains $D^\epsilon_i$. For $i=n+1,\dots,n+m$, since $p^\epsilon=0$ on $\partial D_i^\epsilon$, and the domain $D_i^\epsilon$ is a thin simply-connected domain with width $\epsilon\ll 1$, intuitively one would expect that $| p^\epsilon| \leq C\epsilon^2$. The next proposition shows that this crude estimate is indeed true, and its proof is postponed to Section~\ref{subsec_proof_p}.

\begin{proposition}\label{prop_p_curve}
For any $i=n+1,\dots,n+m$, let $p^\epsilon:\overline{D_i^\epsilon}\to \mathbb{R}$ be given by \eqref{def_p_laplacian}--\eqref{bdry_curve}. Then there exist $\epsilon_1$ and $C$ only depending on $\rVert z_i\rVert_{C^{2}(S_i)},\rVert \gamma_i \rVert_{L^\infty(S_i)}$ and $F_\Gamma$, such that 
\[
|p^\epsilon|\leq C\epsilon^2 \quad\text{ in }D_i^\epsilon
\]
for all $\epsilon \in (0,\epsilon_1)$.
\end{proposition}

For $i=1,\dots, n$, the estimate is more involved, since $p^\epsilon$ takes different values $c^\epsilon_i$ and $0$ on the inner and outer boundaries of $D^\epsilon_i$. Heuristically speaking, since $D^\epsilon_i$ is a doubly-connected thin tubular domain with width $\sim\epsilon$, we would expect that $p^\epsilon_i$ (in $\alpha,\eta$ coordinate) changes almost linearly  from $0$ to $c^\epsilon_i$ as $\eta$ goes from $-1$ (outer boundary) to $0$ (inner boundary). Next we will show that the error between $p^\epsilon(R^\epsilon_i(\alpha,\eta))$ and the linear-in-$\eta$ function $c^\epsilon_i(1+\eta)$ is indeed controlled by $O(\epsilon^2)$. We will also obtain fine estimates of the gradient of the function $c^\epsilon_i(1+\eta)$, as well as the boundary value $c^\epsilon_i$. Again, its proof is postponed to Section~\ref{subsec_proof_p}.

\begin{proposition}\label{P_pdecomposition}
For any $i=1,\dots,n$, let $p^\epsilon:\overline{D_i^\epsilon}\to \mathbb{R}$ and $c^\epsilon_i \in \mathbb{R}$ be given by \eqref{def_p_laplacian} and \eqref{bdry_loop}--\eqref{bdry_loop2}. For such $p^{\epsilon}$, let us define $\tilde{p}^{\epsilon}$,$q^{\epsilon}:\overline{D_i^\epsilon}\mapsto \R$ as follows:
\begin{align}
\tilde{p}^{\epsilon}(R^\epsilon_i(\alpha,\eta)) &:=  c^\epsilon_i(1+\eta) \quad\text{ for }\alpha \in S_i, \eta\in[0,-1], \label{P_ptildedef}\\
q^{\epsilon} &:= p^{\epsilon} - \tilde{p}^{\epsilon}\quad\quad\text{ in }\overline{D^\epsilon_i}. \nonumber
\end{align}
Also let \begin{equation}\label{P_betadef2}
\beta_i := \frac{2|U_i|}{L_i\int_{S_i}\gamma_i^{-1}(\alpha)d\alpha}.
\end{equation}
Then there exist $\epsilon_1$ and $C$ only depending on $\rVert z_i\rVert_{C^{3}(S_i)},\rVert \gamma_i \rVert_{C^2(S_i)}$ and $F_\Gamma$, such that for all $\epsilon\in(0,\epsilon_1)$ we have the following:
\begin{align}
&\begin{cases}
|q^{\epsilon}|\le C \epsilon^2 & \text{ in }D^{\epsilon}_i,\\
q^{\epsilon} = 0 & \text{ on }\partial D^{\epsilon}_i,
\end{cases} \label{P_claim16}\\[0.1cm]
&\left|\frac{c^\epsilon_i}{\epsilon}-\beta_i\right| \leq C\epsilon, \label{c_beta}
\\[0.1cm]
&\left| \nabla \tilde{p}^{\epsilon}(R^\epsilon_i(\alpha,\eta))- \frac{\beta_i}{\gamma_i(\alpha)}\mathbf{n}(z_i(\alpha))\right| \le C\epsilon \quad\text{ for }\alpha \in S_i, \eta\in[0,-1]. \label{P_claim15}
\end{align}
\end{proposition}

\subsection{Proof of the quantitative lemmas for $p^\epsilon$}\label{subsec_proof_p}
In this subsection we aim to prove Propositions~\ref{prop_p_curve} and \ref{P_pdecomposition}. We start with a technical lemma on estimating the solution of Poisson's equation (with zero boundary condition) in the domain $D^\epsilon_i$.
 
\begin{lemma}\label{P_maximumprinciple} For any $i=1,\dots, n+m$, assume $\Gamma_i$ and $\gamma_i$ satisfy \textbf{\textup{(H1)}}--\textbf{\textup{(H3)}}. 
Let $v^\epsilon \in C^2(D^\epsilon_i) \cap C(\overline{D^\epsilon_i})$ solve   the Poisson's equation with zero boundary condition:
\begin{equation}\label{P_poisson}
\begin{cases}
\Delta v^\epsilon = -1 & \text{ in }D^{\epsilon}_i, \\
 v^{\epsilon}  =0 & \text{ on }\partial D^{\epsilon}_i.
\end{cases}
\end{equation}
Then there exist positive constants $\epsilon_0=C(\rVert z_i\rVert_{C^{2}(S_i)},\rVert \gamma_i\rVert_{L^\infty(S_i)},F_\Gamma)$ and $C_1, C_2 = C(\|\gamma_i\|_{L^\infty(S_i)})$, such that for all $\epsilon \in (0,\epsilon_0)$ we have
 \begin{equation}
0 \leq v^\epsilon \leq C_1 \epsilon^2\quad\text{ in } D^\epsilon_i
\label{P_bound1}
\end{equation}
and
\begin{equation}
 \rVert \nabla v^\epsilon\rVert_{L^{\infty}(\Gamma_i)} \le C_2\epsilon \quad \text{ for } i=1, \ldots, n.\label{P_bound2}
\end{equation}
\end{lemma}

\begin{proof} Throughout the proof, let $i\in \{1,\dots,n+m\}$ be fixed. For notational simplicity, in the rest of the proof we omit the subscript $i$ in $R^\epsilon_i$, $D^\epsilon_i$, $S_i$, $z_i$ and $\gamma_i$.

\textbf{Step 1.} We start with a simple geometric result that $D^\epsilon$ is ``flat'' in a small neighborhood of any $z(\alpha)$.
For any $\alpha\in S$, let $V^\epsilon(\alpha):= D^\epsilon \cap B_{6\epsilon \|\gamma\|_\infty}(z(\alpha))$, where $\|\cdot\|_\infty$ denotes $\|\cdot\|_{L^\infty(S)}$. We will show that any $y\in V^\epsilon(\alpha)$ satisfies 
\begin{equation}\label{claim_alpha}
\big|(z(\alpha)-y)\cdot \mathbf{n}(z(\alpha))\big| \leq 2\epsilon\|\gamma\|_\infty
\end{equation}
for all sufficiently small $\epsilon>0$ (to be quantified in \eqref{eps_bound}). See Figure~\ref{fig_barrier}(a) for an illustration.

Since $y\in V^\epsilon(\alpha)\subset D^\epsilon$, there exist $\beta\in S$ and $\eta \in (-1,0)$ such that $y=R^\epsilon(\beta,\eta) = z(\beta) + \epsilon\gamma(\beta) \mathbf{n}(z(\beta))\eta$. 
It follows that
\begin{equation}
\begin{split}
\label{tube}
 \big|(z(\alpha)-y)\cdot \mathbf{n}(z(\alpha))\big| & \le |(z(\alpha)-z(\beta))\cdot \mathbf{n}(z(\alpha))|+\epsilon\rVert \gamma\rVert_{\infty}\\
&\le \rVert z'' \rVert_{\infty}(\alpha-\beta)^2+\epsilon\rVert \gamma\rVert_{\infty},
\end{split}
\end{equation}
where in the second inequality we used
\begin{equation}\label{eq_diff_temp}
|(z(\alpha)-z(\beta)) - z'(\alpha)(\alpha-\beta)| \leq \|z''\|_{\infty} (\alpha-\beta)^2
\end{equation} and $z'(\alpha)\cdot \mathbf{n}(z(\alpha))=0$. To bound $\alpha-\beta$ on the right hand side of \eqref{tube}, the fact that $y\in B_{6\epsilon\|\gamma\|_\infty}(z(\alpha))$ gives
\begin{equation}\label{eq_diff_temp2}
6\epsilon \rVert \gamma \rVert_{\infty} \ge |z(\alpha)-y| \geq |z(\alpha)-z(\beta)| - \epsilon \gamma(\beta),
\end{equation}
which implies
$
|z(\alpha)-z(\beta)| \leq 7 \epsilon \rVert \gamma \rVert_{\infty}
$. Since the arc-chord constant $F_\Gamma$ given in \eqref{def_arc_chord} is finite, this implies 
\begin{equation}\label{eq_diff_temp3}
|\alpha - \beta| \leq 7 F_\Gamma \rVert \gamma \rVert_{\infty} \epsilon.
\end{equation}  Plugging this into the right hand side of \eqref{tube}, we know \eqref{claim_alpha} holds for all 
\begin{equation}\label{eps_bound}
0<\epsilon \leq (49\|z''\|_\infty F_\Gamma^2 \|\gamma\|_\infty)^{-1}.
\end{equation}

\begin{figure}
\begin{center}
\includegraphics[scale=0.9]{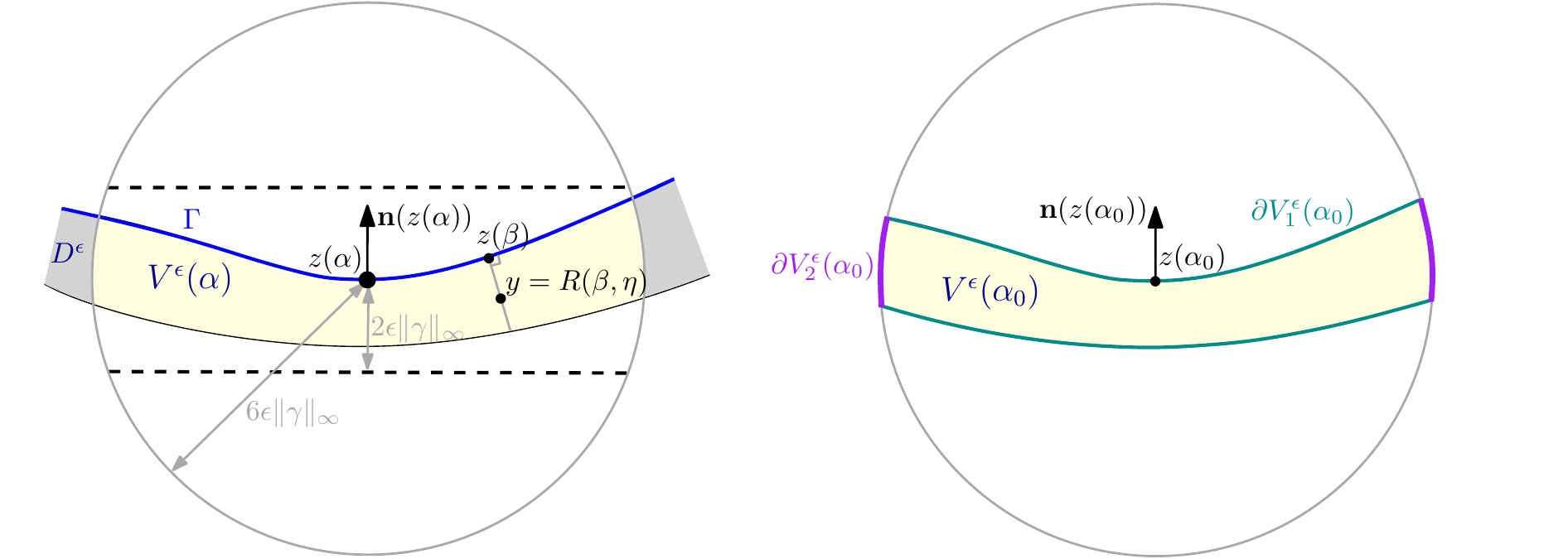}
\hspace*{0.5cm}(a)\hspace{8.3cm}(b)
\caption{(a) In Step 1,  $V^\epsilon(\alpha)$ (the yellow set) must lie between the two dashed lines for small $\epsilon$. (b) In Step 2, $\partial V^\epsilon(\alpha_0)$ is decomposed into $\partial V_1^\epsilon(\alpha_0)$ (in dark green) and $\partial V_2^\epsilon(\alpha_0)$ (in purple). \label{fig_barrier}}
\end{center}
\end{figure}

\textbf{Step 2.} Next we prove \eqref{P_bound1}. Note that $v^\epsilon$ is superharmonic in $D^{\epsilon}$ and vanishes on the boundary, thus   it follows from the maximum principle that $v^\epsilon\geq 0$ in $D^\epsilon$. Denote $M:=\max_{x\in D^{\epsilon}}v(x)$, and pick $x_0=R(\alpha_0,\eta_0) \in D^{\epsilon}$ such that $v(x_0) = M$. 
 Without loss of generality, we can assume that $z(\alpha_0)=(0,0)$ and $\mathbf{s}(z(\alpha_0))=\mathbf{e}_1:=(1,0)$, so that $\mathbf{n}(z(\alpha_0))=(0,1)$ and $x_0 =(0, \epsilon\gamma(\alpha_0)\eta_0)$. Let us consider a barrier function $b_1:\R^2 \mapsto \R$ given by 
\[
b_1(x_1,x_2) = x_2^2 - \frac{x_1^2}{2}.
\]
Clearly $\Delta b_1=1$, so $v^\epsilon+b_1$ is harmonic in $D^\epsilon$. It then follows from the maximum principle that $\max_{\overline{V^\epsilon(\alpha_0)}} (v^\epsilon+b_1)$ is achieved at some boundary point  $\tilde x_0\in \partial V^\epsilon(\alpha_0)$. Let us break $\partial V^\epsilon(\alpha_0)$ into $\partial V^\epsilon_1(\alpha_0) \cup \partial V^\epsilon_2(\alpha_0)$ (see Figure~\ref{fig_barrier}(b) for an illustration), given by 
\begin{equation}\label{bdry_decompose}
\partial V^\epsilon_1(\alpha_0):=\partial D^{\epsilon}\cap B_{6\epsilon\rVert \gamma \rVert_{\infty}}(z(\alpha_0)), \quad \partial V^\epsilon_2(\alpha_0):=\overline{D^{\epsilon}}\cap \partial B_{6\epsilon\rVert \gamma\rVert_{\infty}}(z(\alpha_0)).
 \end{equation}

 We claim that $\tilde{x}_0 \in \partial V^\epsilon_1(\alpha_0)$. To see this, note that any $y=(y_1,y_2)\in \partial V^\epsilon_2(\alpha_0)$ satisfies $|y|=6\epsilon\rVert \gamma \rVert_{\infty}$ and $|y_2|\leq 2\epsilon\rVert \gamma \rVert_{\infty}$, where the latter follows from \eqref{claim_alpha} and our assumption that $\mathbf{s}(z(\alpha_0))= \mathbf{e}_1$. This implies that $|y_1|\geq 4\epsilon\rVert \gamma \rVert_{\infty}>|y_2|$, thus $b_1(y)<0$. Using that $v^\epsilon(x_0) = M \geq v^\epsilon(y)$ and $b_1(x_0)= b_1(0, \epsilon\gamma(\alpha_0)\eta_0)\geq 0$, we have
  $(v^\epsilon+b_1)(y) < (v^\epsilon+b_1)(x_0).$
   This shows that $\max_{\overline{V^\epsilon(\alpha_0)}}(v^\epsilon + b_1) $ cannot be achieved on $\partial  V^\epsilon_2(\alpha_0)$, finishing the proof of the claim.

Since $\tilde{x}_0 \in \partial V^\epsilon_1(\alpha_0) \subset \partial D^\epsilon$, the boundary condition in \eqref{P_poisson} yields that $v^\epsilon(\tilde{x}_0)=0$. Thus
 \begin{align*}
 M+b_1(x_0)= v^\epsilon(x_0)+b_1(x_0) \le v^\epsilon(\tilde{x}_0)+b_1(\tilde{x}_0) = b_1(\tilde{x}_0).
 \end{align*}
 Using  $b_1(x_0)= b_1(0, \epsilon\gamma(\alpha_0)\eta_0)\geq 0$, the above inequality becomes 
 \begin{align}\label{P_case1}
 M\le b_1(\tilde x_0) \leq |\tilde x_0|^2 \leq 36 \rVert \gamma \rVert_{\infty}^2\epsilon^2,
 \end{align}
where the second inequality follows from the definition of $b_1$. This proves \eqref{P_bound1} for $C_1=36 \rVert \gamma \rVert_{\infty}^2$.

\textbf{Step 3.} It remains to prove \eqref{P_bound2}. First note that for $i\in\{1,\dots,n\}$, the assumptions \textbf{(H1)}--\textbf{(H3)} yield that $D_i^\epsilon$ has $C^2$ boundary, therefore $v^\epsilon \in C^2(D_i^\epsilon)\cap C^1(\overline{D_i^\epsilon})$. Let us fix $i\in \left\{ 1, \ldots, n\right\}$ and any $\alpha\in S$, and we aim to show that $|\nabla v^\epsilon(z(\alpha))|\leq C_2 \epsilon$.  Again, without loss of generality we can assume that $z(\alpha)=(0,0)$ and $\mathbf{s}(z(\alpha))=\mathbf{e}_1$. 
Let us consider a new barrier function $b_2:\mathbb{R}^2\to\mathbb{R}$
  \begin{equation}\label{def_b2}
  b_2(x_1,x_2) := x_2^2 + 4\epsilon\rVert \gamma\rVert_{\infty} x_2 - \frac{x_1^2}{2},
  \end{equation}
 which satisfies $b_2(0,0)=0$, and one can easily check that its zero level set has horizontal tangent at $(0,0)$ (thus tangent to $\partial D^\epsilon$ at $z(\alpha)$). 
  
Again, let us decompose $\partial V^\epsilon(\alpha)$  as $\partial V^\epsilon_1(\alpha) \cup \partial V^\epsilon_2(\alpha)$ as in \eqref{bdry_decompose} (except that $\alpha_0$ now becomes $\alpha$). 
We claim that for all sufficiently small $\epsilon>0$, the new barrier function $b_2$ satisfies
\begin{align}
&\Delta b_2 = 1~~~\, \text{ in }V^\epsilon(\alpha), \label{P_claim11}\\
& b_2 \le 0\quad ~~~ \text{ on }\partial V^\epsilon_1, \label{P_claim12}\\
&b_2 \le -\epsilon^2 \quad \text{ on } \partial V^\epsilon_2. \label{P_claim14}
\end{align}

Let us assume for a moment that \eqref{P_claim11}--\eqref{P_claim14} are true. Then it follows that 
\begin{equation}\label{sum_barrier}
v^\epsilon+C_2b_2\le 0 \text{ in } V^\epsilon(\alpha),
\end{equation}
 where $C_2:=\max\left\{1,C_1\right\}$ and $C_1$ is as in \eqref{P_bound1} (in the end of step 2 we have $C_1 = 36 \rVert \gamma \rVert_{\infty}^2$). To show \eqref{sum_barrier}, note that $v^\epsilon+C_2b_2$ is subharmonic in $V^\epsilon(\alpha)$ due to \eqref{P_claim11} and the definition of $C_2$, thus its maximum is attained on its boundary. The boundary conditions in \eqref{P_poisson} and \eqref{P_claim12} yield that $v^\epsilon+C_2 b_2 \le 0$ on $\partial V^\epsilon_1(\alpha)$; whereas \eqref{P_bound1}, \eqref{P_claim14} and the definition of $C_2$ yield that $v^\epsilon+C_2b_2 \le 0$ on $\partial V^\epsilon_2(\alpha)$.  Thus $v^\epsilon+C_2 b_2 \le 0$ on $\partial V^\epsilon_1(\alpha) \cup \partial V^\epsilon_2(\alpha)$, implying \eqref{sum_barrier}.

 However, $v^\epsilon+C_2b_2$ is actually zero at $z(\alpha) \in \partial V^\epsilon(\alpha)$, therefore Hopf's Lemma implies that $\nabla \left( v^\epsilon+C_2b_2\right)(z(\alpha)) \cdot \vec{n}(z(\alpha)) > 0$, where $\vec{n}(z(\alpha))$ is the outer normal of $\partial D^{\epsilon}$ at $z(\alpha)$. Hence  
 \begin{align}\label{P_lowerbdforgradient}
| \nabla v^\epsilon(z(\alpha))|= -\nabla v^\epsilon(z(\alpha))\cdot \vec{n}(z(\alpha)) < C_2\nabla b_2(z(\alpha))\cdot \vec{n}(z(\alpha)) = 4C_2\rVert \gamma \rVert_{\infty}\epsilon,
 \end{align}
 where the first equality follows from the fact that $v^\epsilon$  is superharmonic in $D^{\epsilon}$ and constant on $\partial D^\epsilon$, and the second equality is a direct computation of $\nabla b_2$. Thus \eqref{P_lowerbdforgradient} proves \eqref{P_bound2}.  
 
  To complete the proof, we only need to prove \eqref{P_claim11}--\eqref{P_claim14} for small $\epsilon>0$. Note that \eqref{P_claim11} follows immediately from computing  the  Laplacian of $b_2$.
  For \eqref{P_claim12},  let  us pick $y\in \partial V^\epsilon_1(\alpha)$, and we aim to show that $b_2(y)\leq 0$. Note that $y=R^\epsilon(\beta,0)$ or $R^\epsilon(\beta,-1)$ for some $\beta\in S$.   We first deal with the first case. 

Let us denote $y=(y_1,y_2)$. Rewriting \eqref{eq_diff_temp} into two inequalities for the two components, and using that  $z(\alpha)=(0,0)$ and $z'(\alpha)=L\mathbf{e}_1$ ($L$ is the length of the curve $\Gamma_i$), we have
\begin{align}
& |0-y_1 - L(\alpha-\beta)| \leq \|z''\|_\infty(\alpha-\beta)^2 \label{y1}\\
& |y_2| = |0-y_2| \leq \|z''\|_\infty(\alpha-\beta)^2. \label{y2}
\end{align}
Also, \eqref{eq_diff_temp3} gives $|\alpha-\beta|\leq 7 F_\Gamma \|\gamma\|_\infty\epsilon$. 
  Applying it to \eqref{y1}, for all  $\epsilon>0$ sufficiently small we have that
\begin{equation}\label{y1_2}
|y_1| \geq \frac{L}{2}|\alpha-\beta|.
\end{equation}
Plugging \eqref{y1_2} and \eqref{y2} into $b_2(y) = - \frac{1}{2} y_1^2  + y_2^2 + 4\epsilon\|\gamma\|_\infty y_2 $, we have
  \[
  \begin{split}
  b_2(y) &\leq - \frac{L^2}{8}(\alpha-\beta)^2+ \|z''\|_\infty^2 (\alpha-\beta)^4+ 4\epsilon\|\gamma\|_\infty \|z''\|_\infty(\alpha-\beta)^2   \\
  & \leq \left(-\frac{L^2}{8} + C\epsilon^2 + C\epsilon\right) (\alpha-\beta)^2 \leq 0,
  \end{split}
  \]
  for all $\epsilon>0$ sufficiently small, where the second inequality follows from \eqref{eq_diff_temp3}. This finishes the proof of \eqref{P_claim12} for the case $y=R^\epsilon(\beta,0)$.
  
  Before we deal with the case $y=R^\epsilon(\beta,-1) $, let us prove  \eqref{P_claim14} first. For any $y=(y_1,y_2)\in \partial V^\epsilon_2(\alpha)$, \eqref{claim_alpha} gives $|y_2|\leq 2\epsilon\|\gamma\|_\infty $. Combining this with $|y|=6\epsilon \|\gamma\|_\infty $ yields $|y_1|\geq \sqrt{32}\epsilon\rVert \gamma \rVert_{\infty}$.     Thus
\[  
b_2(y)  \leq (2\epsilon\|\gamma\|_\infty)^2 + 4\epsilon\rVert \gamma\rVert_{\infty} (2\epsilon\|\gamma\|_\infty) - \frac{(\sqrt{32}\epsilon\rVert \gamma \rVert_{\infty})^2}{2} \leq -4\epsilon^2 \|\gamma\|_\infty^2.
 \]
  
 Finally we turn to the proof of \eqref{P_claim12} for the case $y=R^\epsilon(\beta,-1)$. Note that the curve $\{R^\epsilon(\beta,-1):\beta\in S\} \cap B_{6\epsilon \|\gamma\|_\infty}(z(\alpha))$ lies in the interior of the region bounded by $\Gamma \cap B_{6\epsilon \|\gamma\|_\infty}(z(\alpha))$ on the top, $\partial B_{6\epsilon\|\gamma\|_\infty}(z(\alpha))$ on the sides, and $y_2 = -2\epsilon\|\gamma\|_\infty$ on the bottom. (The last one  follows from \eqref{claim_alpha} and our assumption that $\mathbf{s}(z(\alpha))= \mathbf{e}_1$). We have already shown $b_2\leq 0$ on $\Gamma \cap B_{6\epsilon \|\gamma\|_\infty}(z(\alpha))$ and the lateral boundaries, and it is easy to check that $b_2\leq 0$ on $y_2 = -2\epsilon\|\gamma\|_\infty$. Since the set $\{b_2\leq 0\}$ is simply-connected, it implies that $b_2\leq 0$ in the interior of this region, finishing the proof.
\end{proof}

Note that \eqref{P_bound1} of Lemma~\ref{P_maximumprinciple} immediately implies Proposition~\ref{prop_p_curve}. (The only difference is that $\Delta v^\epsilon=-1$ in Lemma~\ref{P_maximumprinciple} whereas $\Delta p^\epsilon=-2$ in Proposition~\ref{prop_p_curve}, so the constant $C$ in Proposition~\ref{prop_p_curve} is twice of that in \eqref{P_bound1}). The lemma also implies the following corollary, which will be helpful in the proof of Proposition~\ref{P_pdecomposition}.

\begin{corollary}\label{P_maximumprinciple2}
For any $i=1,\dots, n+m$, assume $\Gamma_i$ and $\gamma_i$ satisfy \textbf{\textup{(H1)}}--\textbf{\textup{(H3)}}. Assume $v^\epsilon \in C^2(D^\epsilon_i) \cap C(\overline{D^\epsilon_i})$ satisfies that
\[ \begin{cases}
 \left| \Delta v^\epsilon \right| \le C_0  & \text{ in }D^{\epsilon}_i, \\
 v^\epsilon = 0 & \text{ on }\partial D^{\epsilon}_i,
 \end{cases}
\] for some constant  $C_0>0$.
 Then for the same constants $\epsilon_0, C_1, C_2$ as in Lemma~\ref{P_maximumprinciple}, the following holds for all $\epsilon \in (0,\epsilon_0)$:
 \begin{equation}
 |v^\epsilon| \leq C_0 C_1 \epsilon^2\quad\text{ in } D^\epsilon_i,
\label{P_bound4}
\end{equation}
and if $v^\epsilon \in C^2(D^\epsilon_i) \cap C^1(\overline{D^\epsilon_i})$, we have
\begin{equation}
 \rVert \nabla v^\epsilon\rVert_{L^{\infty}(\Gamma_i)} \le C_0 C_2 \epsilon \quad \text{ for }\quad i=1,\ldots,n .\label{P_bound5}
 \end{equation}
\end{corollary}

\begin{proof}
Let $\tilde{v}$ be a solution to 
\[\begin{cases}
\Delta \tilde{v} = -C_0 & \text{ in }D^{\epsilon}_i, \\
\tilde{v} = 0 & \text{ on }\partial D^{\epsilon}_i.
\end{cases}
\]
It is clear that $v^\epsilon+\tilde{v}$ is super-harmonic and $v^\epsilon-\tilde{v}$ is sub-harmonic in $D^{\epsilon}_i$, and they both vanish on the boundary. Thus the maximum principle implies that 
\begin{equation}\label{P_bound6}
-\tilde{v} \le v^\epsilon \le \tilde{v}\quad\text{ in }D^\epsilon_i.
\end{equation}
 Applying \eqref{P_bound1} of Lemma~\ref{P_maximumprinciple} to $\frac{\tilde{v}}{C_0}$, we obtain $0\leq \tilde v \leq C_0 C_1\epsilon^2$ in $D^\epsilon_i$ for all $\epsilon \in (0,\epsilon_0)$, leading to \eqref{P_bound4}. Furthermore,  \eqref{P_bound6} and the fact that $v^\epsilon$ and $v$ both have zero boundary condition imply that
\[
 |\nabla v^\epsilon |\le |\nabla \tilde{v}|\quad \text{ on } \partial D^\epsilon_i.
\]
We then apply \eqref{P_bound2} of Lemma~\ref{P_maximumprinciple} to $\frac{\tilde{v}}{C_0}$ and obtain $\rVert\nabla v^\epsilon  \rVert_{L^{\infty}(\Gamma_i)} \leq C_0 C_2\epsilon,$ which proves \eqref{P_bound5}.
\end{proof}


Now we are ready to prove Proposition~\ref{P_pdecomposition}.
\begin{proof}[\textbf{\textup{Proof of Proposition~\ref{P_pdecomposition}}}]
Throughout the proof, let $i\in \{1,\dots,n\}$ be fixed. For notational simplicity, in the rest of the proof we omit the subscript $i$ from all terms.

We claim that
\begin{align}
&\left|\nabla \tilde{p}^{\epsilon}(R^\epsilon(\alpha,\eta)) 
-\frac{c^\epsilon}{\epsilon\gamma(\alpha)}\mathbf{n}(z(\alpha)) \right| \leq C\epsilon \quad\text{ for all }\alpha\in S, \eta \in [0,-1], \label{P_claim21} \\
& \rVert \Delta q^{\epsilon}\rVert_{L^{\infty}(D^{\epsilon})} \le C \label{P_claim22}
\end{align}
for some constant $C>0$ only depending on $\rVert z_i\rVert_{C^{3}(S_i)},\rVert \gamma_i \rVert_{C^2(S_i)}$ and $F_\Gamma$.
Assuming these are true, let us explain how they lead to \eqref{P_claim16}--\eqref{P_claim15}. 
By \eqref{bdry_loop} and \eqref{P_ptildedef}, $p^\epsilon$ and $\tilde p^\epsilon$ have the same boundary condition, thus $q^{\epsilon}=0$ on $\partial D^{\epsilon}$. This and \eqref{P_claim22} allow us to apply Corollary~\ref{P_maximumprinciple2} to $q^\epsilon$ to obtain the estimate \eqref{P_bound4}, implying \eqref{P_claim16}. 

Due to \eqref{P_bound5} of Corollary~\ref{P_maximumprinciple2}, we also have
\begin{equation}\label{q_grad}
\rVert \nabla q^{\epsilon} \rVert_{L^{\infty}( \Gamma)} \leq C\epsilon.
\end{equation}
Using \eqref{bdry_loop2} and $p^\epsilon = \tilde p^\epsilon + q^\epsilon$, we have
\begin{align*}
-2|U| & = \int_{\partial U} \nabla \tilde{p}^{\epsilon}\cdot {n}d\sigma + \int_{\partial U}\nabla q^{\epsilon}\cdot nd\sigma \\
& = -\frac{c^{\epsilon}L}{\epsilon}\int_{S}\gamma^{-1}(\alpha)d\alpha + O(\epsilon),
\end{align*}
where the second equality follows from \eqref{P_claim21} for $\eta=0$, $n(z(\alpha))=-\mathbf{n}(z(\alpha))$ and $d\sigma = L d\alpha$, as well as \eqref{q_grad}. Rearranging the terms and using the definition of $\beta$ in \eqref{P_betadef2} yields \eqref{c_beta}.


Finally, note that \eqref{c_beta} and \eqref{P_claim21} directly lead to \eqref{P_claim15}, where we are using the fact that $\gamma_i$ is uniformly positive for $i=1,\dots,n$, due to \textbf{(H3)}.

The rest of the proof is devoted to proving the claims \eqref{P_claim21} and \eqref{P_claim22}. For \eqref{P_claim21}, we compute the gradient of $\tilde{p}^{\epsilon}$. Differentiating \eqref{P_ptildedef} with respect to $\alpha$ and $\eta$, we obtain
\begin{align}\label{P_derivative}
(\nabla_{\alpha,\eta}R^\epsilon(\alpha,\eta))^{t}\nabla \tilde{p}(R^\epsilon(\alpha,\eta)) = 
\begin{pmatrix}
0 \\
c^{\epsilon}
\end{pmatrix},
\end{align}
where $(\nabla_{\alpha,\eta}R^\epsilon)^{t}$ denotes the transpose of the Jacobian matrix of $R^\epsilon$.  Since $\nabla_{\alpha,\eta}R^\epsilon = ( \partial_\alpha R^\epsilon,  \partial_\eta R^\epsilon)$, using the formula for inverses of $2\times 2$ matrices, we have
\begin{equation}\label{P_inverseofjacobian}
\left(\left(\nabla_{\alpha,\eta}R^\epsilon\right)^{t}\right)^{-1} = \frac{1}{J(\alpha,\eta)}\big(\!-\!(\partial_\eta R^\epsilon)^{\perp},(\partial_\alpha R^\epsilon)^{\perp}\big). 
 \end{equation}
where $J(\alpha,\eta) := \text{det}(\nabla_{\alpha,\eta}R^\epsilon)$.
Multiplying the inverse matrix on both sides of \eqref{P_derivative}, we have
\begin{equation}\label{P_gradient}
\nabla \tilde{p}^{\epsilon}(R^\epsilon(\alpha,\eta)) = 
\frac{1}{J}
\big(\!-\!(\partial_\eta R^\epsilon)^{\perp},(\partial_\alpha R^\epsilon)^{\perp}\big)
\begin{pmatrix}
0\\
c^{\epsilon}
\end{pmatrix}
=\frac{c^{\epsilon}}{J} (\partial_\alpha R^\epsilon)^{\perp}.
\end{equation}
Recall that Lemma~\ref{P_derivatives1} gives $(\partial_\alpha R^\epsilon)^{\perp} = z'(\alpha)^\perp+O(\epsilon) = L\mathbf{n}(z(\alpha))+O(\epsilon)$, and $J= \epsilon L \gamma + O(\epsilon^2)$. Plugging these into 
\eqref{P_gradient} gives 
\begin{equation}\label{p_grad_2}
\nabla \tilde{p}^{\epsilon}(R(\alpha,\eta))=\frac{c^\epsilon}{\epsilon}\left(\frac{\mathbf{n}(z(\alpha))}{\gamma} + O(\epsilon)\right).
\end{equation}
 Note that it follows from \eqref{ineq_sup} that $c^{\epsilon} \le \frac{|D^{\epsilon}|}{2\pi}$, where $|D^\epsilon|\leq C\epsilon$ due to \eqref{P_nonzero}. These imply 
 \begin{equation}\label{bd_ce}
 \frac{c_\epsilon}{\epsilon}\leq C,
 \end{equation} and applying it to \eqref{p_grad_2} yields \eqref{P_claim21}.

To prove \eqref{P_claim22}, since $q^\epsilon=p^\epsilon-\tilde p^\epsilon$ and $\Delta p^\epsilon=-2$ in $D^\epsilon$, it suffices to show that
\begin{equation}\label{P_claim23}
\left |\Delta \tilde{p}^{\epsilon} \right|\le C\quad\text{ in } D^\epsilon,
\end{equation}
and we will begin with an explicit computation of $\partial_{x_1x_1} \tilde p^\epsilon$ and $\partial_{x_2x_2} \tilde p^\epsilon$. Let us denote $R^\epsilon =: (R^1, R^2)$. For notational simplicity, in the rest of the proof we will use subscripts on $R^\epsilon$, $R^1$, $R^2$ and $J$ to denote their partial derivative, e.g.  $ R^1_\alpha := \partial_\alpha R^1$.

 From \eqref{P_gradient}, it follows that 
\[\partial_{x_1}\tilde{p}^\epsilon(R^\epsilon(\alpha,\eta)) = -\frac{c^{\epsilon}}{J} R^2_\alpha.
\]
Differentiating in $\alpha$ and $\eta$, we get
\[
\begin{split}\nabla\left(\partial_{x_1}\tilde{p}^{\epsilon}\right)(R^\epsilon(\alpha,\eta)) & = \left( \left( \nabla_{\alpha,\eta}R^\epsilon\right)^t\right)^{-1}\nabla_{\alpha,\eta}\left(-\frac{c^{\epsilon}}{J}R^2_\alpha\right)\\
& = \frac{c^{\epsilon}}{J}\begin{pmatrix} R_\eta^2 & -R_\alpha^2\\[0.1cm] -R_\eta^1 & R_\alpha^1\end{pmatrix}
\begin{pmatrix}
\frac{J_\alpha}{J^2} R^2_\alpha - \frac{1}{J}R^2_{\alpha\alpha}\\[0.1cm]
\frac{J_\eta}{J^2} R_\alpha^2-\frac{1}{J}R^{2}_{\alpha\eta}
\end{pmatrix},
\end{split}
\]
thus 
\[
\partial_{x_1x_1} \tilde{p}^{\epsilon}(R(\alpha,\eta)) = \frac{c^\epsilon}{J}\left(\frac{ J_\alpha}{J^2} R^2_\eta  R^2_\alpha - \frac{1}{J} R^2_\eta R^2_{\alpha\alpha}- \frac{J_\eta}{J^2} (R_\alpha^2)^2 + \frac{1}{J}R^2_\alpha R^2_{\alpha\eta}\right).
\] 
Likewise, $\partial_{x_2x_2}\tilde{p}(R(\alpha,\eta))$ takes the same expression except every $R^2$ is changed into $R^1$. Adding them together gives
\begin{equation}\label{delta_p}
\Delta \tilde{p}^{\epsilon}(R(\alpha,\eta)) = \frac{c^\epsilon}{J}\left(\frac{ J_\alpha}{J^2} R^\epsilon_\eta \cdot R^\epsilon_\alpha -\frac{1}{J} R_\eta^\epsilon\cdot R^\epsilon_{\alpha\alpha} - \frac{J_\eta}{J^2} R^\epsilon_\alpha \cdot R^\epsilon_\alpha+ \frac{1}{J} R^\epsilon_\alpha \cdot R^\epsilon_{\alpha\eta} \right).
\end{equation}
Using the explicit  formulae of $R_\alpha, R_\eta$ and $J$ in Lemma~\ref{P_derivatives1}, we directly obtain $|R^\epsilon_\alpha|, |R^\epsilon_{\alpha\alpha}|\leq C$; ~$|R^\epsilon_\eta|, |R^\epsilon_{\alpha\eta}|, |J_\alpha|\leq C\epsilon$; ~$|J_\eta|\leq C\epsilon^2$; ~and $J^{-1}\leq C\epsilon^{-1}$ when $\epsilon$ is sufficiently small, where $C$ depends on $\rVert z_i\rVert_{C^{3}(S_i)}$ and $\rVert \gamma_i \rVert_{C^2(S_i)}$. As a result, all the four terms in the parenthesis of \eqref{delta_p} are bounded by some constant $C$ independent of $\epsilon$. Finally, \eqref{bd_ce} yields $\frac{c_\epsilon}{J}\leq C$ as well, thus  $|\Delta \tilde{p}^{\epsilon}| \le C$, and this proves the second claim \eqref{P_claim22}.
 \end{proof}

\section{Proof of the symmetry result}\label{sec5}
In this section we prove that  a  stationary vortex sheet with positive vorticity must be radially symmetric up to a translation, and a rotating vortex sheet with positive vorticity and angular velocity $\Omega<0$ must be radially symmetric.
The key idea of the proof is to define the integral 
\begin{equation}\label{def_i_eps}
\begin{split}
I^\epsilon &:= \int_{D^\epsilon} \epsilon^{-1} \mathbf{u}^\epsilon \cdot \nabla\left(\omega^\epsilon * \mathcal{N}-\frac{\Omega}{2}|x|^2\right)dx \\&= \int_{D^\epsilon} \epsilon^{-1} (x+\nabla p^\epsilon) \cdot \nabla\left(\omega^\epsilon * \mathcal{N}-\frac{\Omega}{2}|x|^2\right)dx ,
\end{split}
\end{equation}
and compute it in two different ways. 
The motivation of the definition is as follows. As discussed in \cite[Section 2.1]{GomezSerrano-Park-Shi-Yao:radial-symmetry-stationary-solutions}, $I^\epsilon$ can be thought of as a first variation of an ``energy functional'' 
\[
\mathcal{E}[\omega^\epsilon]:=\int \frac{1}{2}\omega^\epsilon (\omega^\epsilon*\mathcal{N})-\frac{\Omega}{2}\omega^\epsilon |x|^2  \,dx
\] when we perturb $\omega^\epsilon$ by a divergence free vector $\mathbf{u}^\epsilon$ in $D^\epsilon$. 
(This functional $ \mathcal{E}$ only serves as our motivation, and will not appear in the proof.) 
 On the one hand, using that $\omega_0$ is stationary in the rotating frame with angular velocity $\Omega$ and $\omega^\epsilon$ is a close approximation of $\omega_0$, we will show in Proposition~\ref{prop1_i} that $I^\epsilon$ is of order $O(\epsilon|\log\epsilon|)$, thus goes to zero as $\epsilon\to 0$. On the other hand,  using the particular $\mathbf{u}^\epsilon$ that we constructed in Section~\ref{subsec_p}, we will prove in Proposition~\ref{prop2_i} that if $\Omega=0$, $I^\epsilon$ is strictly positive independently of $\epsilon$ unless all the vortex sheets are nested circles with constant density; and also prove a similar result in Corollary~\ref{cor2_i} for $\Omega<0$.

\begin{proposition}\label{prop1_i}
Assume $\omega(\cdot,t)=\omega_0(R_{\Omega t}\cdot)$ is a stationary/uniformly-rotating vortex sheet with angular velocity $\Omega\in\mathbb{R}$, where $\omega_0$ satisfies \textbf{\textup{(H1)}}--\textbf{\textup{(H3)}}. Then there exists some $C>0$ only depending on $b$ (as in \textbf{\textup{(H2)}}), $\max_i \|z_i\|_{C^3(S_i)}$, $\max_{i\leq n}\|\gamma_i\|_{C^2(S_i)}$, $\max_{i>n}\|\gamma_i\|_{C^b(S_i)}$, $d_\Gamma$ and $F_\Gamma$, such that $| I^\epsilon | < C \epsilon^b |\log\epsilon|$ for all sufficiently small $\epsilon > 0$. 
\end{proposition}

\begin{proof}

Let us decompose $I^\epsilon =: \sum_{i=1}^{n+m} I^\epsilon_i$, where $I_i^\epsilon :=  \int_{D^\epsilon_i} \epsilon^{-1} (x+\nabla p^\epsilon) \cdot \nabla(\omega^\epsilon * \mathcal{N}-\frac{\Omega}{2}|x|^2)dx .$

We start with showing that $|I^\epsilon_i|\leq C\epsilon^b |\log\epsilon|$ for $i=n+1,\dots, n+m$. For such $i$, $p^\epsilon=0$ on $\partial D^\epsilon_i$, thus the divergence theorem (and the fact that $\omega^\epsilon = \epsilon^{-1}$ in $D^\epsilon_i$) gives
\[
I^\epsilon_i = \underbrace{\int_{D^\epsilon_i} \epsilon^{-1} x \cdot \nabla\left(\omega^\epsilon*\mathcal{N}-\frac{\Omega}{2}|x|^2\right) dx}_{=:T_i^\epsilon}- \int_{D^\epsilon_i} \epsilon^{-1}(\epsilon^{-1}-2\Omega) p^\epsilon(x) dx.
\]
Using the estimate $|p^\epsilon|\leq C\epsilon^2$ in Proposition~\ref{prop_p_curve} and the fact that $|D^\epsilon_i|\leq C\epsilon$ from \eqref{P_nonzero}, we easily bound the second integral by $C\epsilon$. To control the first integral $T_i^\epsilon$, we rewrite it using the change of variables $x = R^\epsilon_i(\alpha,\eta)$ and the definition  $\mathbf{v}^\epsilon := \nabla^\perp( \omega^\epsilon*\mathcal{N})$ in \eqref{def_v_eps}: (also note that on the right hand side we group $\epsilon^{-1}$ with the determinant)
\[
\begin{split}
T_i^\epsilon = \int_{S_i} \int_{-1}^0  R_i^\epsilon(\alpha,\eta) \cdot \underbrace{\Big(\!-\!(\mathbf{v}^\epsilon)^\perp(R_i^\epsilon(\alpha,\eta))-\Omega R_i^\epsilon(\alpha,\eta)\Big)}_{=:J_i^\epsilon} \underbrace{\epsilon^{-1} \det(\nabla_{\alpha,\eta} R_i^\epsilon(\alpha,\eta))}_{=:K_{i}^\epsilon} \,d\eta d\alpha.
\end{split}
\]
Let us take a closer look at the integrand, which is a product of 3 terms. Clearly, the definition of $R_i^\epsilon$ gives $
R_i^\epsilon(\alpha,\eta) = z_i(\alpha) + O(\epsilon).
$
As for the middle term $J_i^\epsilon$, Lemma~\ref{lemma_v} yields
\begin{equation}\label{J_i_eps}
J_i^\epsilon(\alpha,\eta)= - BR^\perp(z_i(\alpha)) + \left(\eta+\frac{1}{2}\right)[\mathbf{v}]^\perp(z_i(\alpha)) - \Omega z_i(\alpha) + O(\epsilon^b |\log\epsilon|).
\end{equation}
Using the fact that $BR(z_i(\alpha))=\Omega z_i^\perp(\alpha)$ for $i=n+1,\dots,n+m$ (which follows from \eqref{BR1} and \eqref{BR2}), it becomes
\begin{equation}\label{J_i_eps_2}
J_i^\epsilon(\alpha,\eta)= \left(\eta+\frac{1}{2}\right)[\mathbf{v}]^\perp(z_i(\alpha))  + O(\epsilon^b |\log\epsilon|).
\end{equation}
Also it follows from \eqref{eq_det} that
$K_i^\epsilon(\alpha,\eta) = L_i \gamma_i(\alpha) + O(\epsilon).
$ Plugging these three estimates into the above integral gives
\[
T_i^\epsilon= \int_{S_i} \int_{-1}^0 z_i(\alpha) \cdot\left(\eta+\frac{1}{2}\right)[\mathbf{v}]^\perp(z_i(\alpha)) L_i \gamma_i(\alpha) d\eta d\alpha + O(\epsilon^b |\log\epsilon|) = O(\epsilon^b |\log\epsilon|),
\]
where the last step follows from the fact that $\int_{-1}^0 (\eta+\frac{1}{2}) d\eta = 0$. This finishes the proof that $|I^\epsilon_i |\leq C\epsilon^b |\log\epsilon|$ for $i=n+1,\dots,n+m$, where $C$ depends on $b$, $\max_i\|z_i\|_{C^2(S_i)}$, $\max_{i}\|\gamma_i\|_{C^b(S_i)}$,  $d_\Gamma$ and $F_\Gamma$.

In the rest of the proof we aim to show $|I^\epsilon_i|\leq C\epsilon^b |\log\epsilon|$ for $i=1,\dots, n$, which is slightly more involved. Recall that in Proposition~\ref{P_pdecomposition} we defined $\tilde p^\epsilon$ and $q^\epsilon$ in $D_i^\epsilon$ for $i=1,\dots, n$, where they satisfy $p^\epsilon = \tilde p^\epsilon + q^\epsilon$ in $D_i^\epsilon$, and  $q^\epsilon = 0$ on $\partial D_i^\epsilon$. This allows us to apply the divergence theorem (to the $q^\epsilon$ term only) and decompose $I_i^\epsilon$ as
\[
I^\epsilon_i = \int_{D^\epsilon_i} \epsilon^{-1} (x+\nabla \tilde p_\epsilon) \cdot \nabla\left(\omega^\epsilon*\mathcal{N}-\frac{\Omega}{2}|x|^2\right) dx- \int_{D^\epsilon_i} \epsilon^{-1}(\epsilon^{-1}-2\Omega) q^\epsilon(x) dx =: I^\epsilon_{i,1} + I^\epsilon_{i,2}.
\]
We can easily show that $I^\epsilon_{i,2} = O(\epsilon)$: \eqref{P_claim16} of Proposition~\ref{P_pdecomposition} gives $|q^\epsilon|\leq C\epsilon^2$, and combining it with $|D^\epsilon_i|\leq C\epsilon$ in \eqref{P_nonzero} immediately yields the desired estimate.

Next we turn to $I^\epsilon_{i,1}$. Again, the change of variables $x = R^\epsilon_i(\alpha,\eta)$ and the definition  $\mathbf{v}^\epsilon := \nabla^\perp( \omega^\epsilon*\mathcal{N})$ gives
\[
I_{i,1}^\epsilon =  \int_{S_i} \int_{-1}^0  \big(R_i^\epsilon(\alpha,\eta)+\nabla\tilde p^\epsilon(R_i^\epsilon(\alpha,\eta))\big) \cdot \underbrace{\Big(\!-\!(\mathbf{v}^\epsilon)^\perp(R_i^\epsilon(\alpha,\eta))-\Omega R_i^\epsilon(\alpha,\eta)\Big)}_{=:J_i^\epsilon} \underbrace{\epsilon^{-1} \det(\nabla_{\alpha,\eta} R_i^\epsilon(\alpha,\eta))}_{=:K_{i}^\epsilon} \,d\eta d\alpha.
\]
For the three terms in the product of the integrand, we will approximate the first term using the definition of $R_i^\epsilon$ and \eqref{P_claim15} of Proposition~\ref{P_pdecomposition}:
\[
R_i^\epsilon(\alpha,\eta)+\nabla\tilde p^\epsilon(R_i^\epsilon(\alpha,\eta)) = z_i(\alpha) + \frac{\beta_i}{\gamma_i(\alpha)}\mathbf{n}(z_i(\alpha)) + O(\epsilon),
\]
where $\beta_i := \frac{2|U_i|}{L_i\int_{S_i}\gamma_i^{-1}(\alpha)d\alpha}$ is given by \eqref{P_betadef2}. Lemma~\ref{lemma_v} allows us to approximate the middle term $J_i^\epsilon$ as \eqref{J_i_eps},  however \eqref{J_i_eps_2} no longer holds since for $i=1,\dots,n$ we do not have $BR(z_i(\alpha))=\Omega z_i^\perp(\alpha)$. As for $K_i^\epsilon$, we again use \eqref{eq_det} to approximate it by
$K_i^\epsilon(\alpha,\eta) = L_i \gamma_i(\alpha) + O(\epsilon).
$
Plugging these three estimates into the integrand of $I_{i,1}^\epsilon$ gives
\[
I_{i,1}^\epsilon =  \int_{S_i} \left(z_i(\alpha) + \frac{\beta_i}{\gamma_i(\alpha)}\mathbf{n}(z_i(\alpha))\right) \cdot \Big(-\!\!BR^\perp(z_i(\alpha)) - \Omega z_i(\alpha)\Big) L_i\gamma_i(\alpha) d\alpha + O(\epsilon^b |\log\epsilon|),
\]
where we again use the fact that the $(\eta+\frac{1}{2})$ term gives zero contribution since $\int_{-1}^0 (\eta+\frac{1}{2}) d\eta =0$.
Next we will show the integral on the right hand side is in fact 0. Since $\omega$ is a rotating solution with angular velocity $\Omega$, the conditions \eqref{BR1} and \eqref{BR2}  yield that 
\[
-BR^\perp(z_i(\alpha)) - \Omega z_i(\alpha)= C_i \gamma_i^{-1}(\alpha)\mathbf{n}(z_i(\alpha)),
\] for some constant $C_i$. Plugging this into the above integral gives
\[
\begin{split}
I_{i,1}^\epsilon &= C_i L_i \int_{S_i} \left(z_i(\alpha) \cdot \mathbf{n}(z_i(\alpha)) + \frac{\beta_i}{\gamma_i(\alpha)}\right) d\alpha + O(\epsilon^b |\log\epsilon|) \\&
= C_i L_i\left( \int_{S_i} z_i(\alpha) \cdot \mathbf{n}(z_i(\alpha)) d\alpha + \frac{2|U_i|}{L_i}\right) + O(\epsilon^b |\log\epsilon|),
\end{split}
\]
where the second step follows from the definition of $\beta_i$ in \eqref{P_betadef2}. Let us compute the integral on the right hand side by changing to arclength parametrization and applying the divergence theorem:
\[
\int_{S_i} z_i(\alpha) \cdot \mathbf{n}(z_i(\alpha)) d\alpha= -\frac{1}{L_i}\int_{\partial U_i} x \cdot n d\sigma = -\frac{2|U_i|}{L_i},
\]
which yields $I_{i,1}^\epsilon = O(\epsilon^b |\log\epsilon|)$, and finishes the proof that $|I_i^\epsilon|\leq C\epsilon^b |\log\epsilon|$ for $i=1,\dots, n$, where $C$ depends on $b$, $\rVert z_i\rVert_{C^{3}(S_i)},\rVert \gamma_i \rVert_{C^2(S_i)}$, $d_\Gamma$ and $F_\Gamma$.

Finally, summing the $I_i^\epsilon$ estimates for $i=1,\dots, n+m$ gives 
$
|I^\epsilon| \leq C\epsilon^b |\log\epsilon|
$
for all sufficiently small $\epsilon>0$, thus we can conclude.
\end{proof}

Now we will use a different way to compute $I^\epsilon$. Let us first define a new integral $\tilde I^\epsilon$ that is the same as $I^\epsilon$ except with $\Omega$ set to zero:
\begin{equation}\label{def_i_eps_tilde}
\begin{split}
\tilde I^\epsilon &:=\int_{D^\epsilon} \epsilon^{-1} (x+\nabla p^\epsilon) \cdot \nabla\left(\omega^\epsilon * \mathcal{N}\right)dx.
\end{split}
\end{equation}
Next we will prove that $\tilde I^\epsilon$ is strictly positive independently of $\epsilon$ unless all the vortex sheets are nested circles with constant density. As we will see in the proof, the key step is to show that if some $\Gamma_i$ is either not a circle or does not have a constant $\gamma_i$, then the estimates on $p^\epsilon$ in Propositions~\ref{prop_p_curve}--\ref{P_pdecomposition} lead to the following quantitative version of \eqref{ineq_int}: $\epsilon^{-2}\left(\frac{|D_i^\epsilon|^2}{4\pi}-\int_{D_i^\epsilon} p^\epsilon(x) dx\right)\geq c_0>0$, where $c_0$ is independent of $\epsilon$.

\begin{proposition}\label{prop2_i} Let $\tilde I^\epsilon$ be defined as in \eqref{def_i_eps_tilde}. Assume that $\Gamma_i$ and $\gamma_i$ satisfy \textbf{\textup{(H1)}}--\textbf{\textup{(H3)}} for $i=1,\dots,n+m$. Then we have
$
\tilde I^\epsilon\geq 0
$ for all sufficiently small $\epsilon>0$.

In addition, if $\Gamma$ is \textbf{not} a union of nested circles with constant $\gamma_i$'s on each connected component, there exists some $c_0>0$ independent of $\epsilon$, such that $\tilde I^\epsilon > c_0 > 0$ for all sufficiently small $\epsilon > 0$.
\end{proposition}

\begin{proof}
We start by decomposing $\tilde I^{\epsilon}$ as 
\[
\tilde I^{\epsilon} = \int_{D^{\epsilon}} \epsilon^{-1}x\cdot \nabla(\omega^{\epsilon}*\mathcal{N})dx + \int_{D^{\epsilon}}\epsilon^{-1}\nabla p^\epsilon\cdot \nabla(\omega^{\epsilon}*\mathcal{N})dx=:I^{\epsilon}_1 + I^{\epsilon}_2.
\] 
$I^\epsilon_1$ can be easily computed as 
\begin{equation}\label{I1}
\begin{split}
I^\epsilon_1  &= \frac{1}{2\pi\epsilon^2} \int_{D^\epsilon}\int_{D^\epsilon}\frac{x\cdot(x-y)}{|x-y|^2}dxdy  
=\frac{|D^\epsilon|^2}{4\pi\epsilon^2} = \frac{1}{4\pi\epsilon^2}\left(\sum_{i=1}^{n+m} |D^\epsilon_i|\right)^2
\end{split}
\end{equation}
where the second equality is obtained by exchanging $x$ with $y$ and taking the average with the original integral. As for $I^\epsilon_2$, we have
\begin{equation}\label{I21}
\begin{split}
 I^\epsilon_2 &=  \frac{1}{\epsilon}\int_{\partial D^\epsilon}p^\epsilon\nabla(\omega^\epsilon * \mathcal{N})\cdot nd\sigma - \frac{1}{\epsilon}\int_{D^\epsilon}p^\epsilon \omega^\epsilon dx \\
 & = -\frac{1}{\epsilon}\sum_{i=1}^nc^\epsilon_i\int_{\partial U_i}\nabla(\omega^\epsilon * \mathcal{N})\cdot nd\sigma - \frac{1}{\epsilon^2}\int_{D^\epsilon}p^\epsilon dx \\
& \geq -\frac{1}{\epsilon^2}\sum_{i=1}^n\sum_{j=1}^{n+m}\frac{|D^\epsilon_i|}{2\pi}\int_{U_i}1_{D^\epsilon_j} dx - \frac{1}{\epsilon^2}\sum_{i=1}^{n+m}\int_{D^\epsilon_i}p^\epsilon dx,
\end{split}
\end{equation} where the first equality follows  from the divergence theorem, the second equality follows from the boundary conditions \eqref{bdry_curve} and \eqref{bdry_loop} for $p^\epsilon$ (as well as the fact that $\partial U_i$ and $\partial D^\epsilon_i$ have opposite outer normals), and the last inequality follows from the divergence theorem as well as the inequality $c_i^\epsilon \leq \sup_{D^\epsilon_i}p \leq \frac{|D^\epsilon_i|}{2\pi}$ due to \eqref{ineq_sup}.

Let us denote $j\prec i$ if $i\in \left\{ 1,\ldots,n\right\}, j\in \{1,\dots, n+m\}$, $j\neq i$ and $\Gamma_j$ lies in the interior of the domain enclosed by $\Gamma_i$ (that is, $\Gamma_j \subset U_i$). If not, we denote $j\nprec i$. Note that for sufficiently small $\epsilon>0$, we have
\begin{equation}\label{j<i}
 \int_{U_i} 1_{D^\epsilon_j} dx =
 \begin{cases}
 |D^\epsilon_j| & \text{ if } j \prec i, \\
 0 & \text{ otherwise.}
 \end{cases} 
\end{equation}
Applying this to \eqref{I21} yields
 \begin{equation}\label{P_i2}
 \begin{split}
 I^\epsilon_2 &\geq -\frac{1}{2\pi \epsilon^2} \sum_{i,j=1}^{n+m}\mathbbm{1}_{j\prec i}\,  |D^\epsilon_i|  |D^\epsilon_j| - \frac{1}{\epsilon^2}\sum_{i=1}^{n+m}\int_{D^\epsilon_i}p^\epsilon_i dx\\
 &= -\frac{1}{4\pi \epsilon^2} \sum_{i,j=1}^{n+m}(\mathbbm{1}_{j\prec i}+ \mathbbm{1}_{i\prec j})\,  |D^\epsilon_i|  |D^\epsilon_j| - \frac{1}{\epsilon^2}\sum_{i=1}^{n+m}\int_{D^\epsilon_i}p^\epsilon_i dx
 \end{split}
 \end{equation}
 where in the first step we used that the $i=n+1,\dots,n+m$ terms have zero contribution in the first sum, due to the definition of $j\prec i$.

 Adding \eqref{I1} and \eqref{P_i2} together, we obtain
 \begin{equation}\label{P_i}
 \tilde I^\epsilon \geq \sum_{i=1}^{n+m} \underbrace{\frac{1}{\epsilon^2}\left(\frac{|D^\epsilon_i|^2}{4\pi}-\int_{D^\epsilon_i}p^\epsilon_idx\right)}_{=:A_i^\epsilon} 
  + \sum_{i,j=1}^{n+m}\underbrace{\frac{1}{\epsilon^2}\left(\mathbbm{1}_{i\neq j}-\left(\mathbbm{1}_{j\prec i}+\mathbbm{1}_{i\prec j}\right)\right)\frac{|D^\epsilon_i||D^\epsilon_j|}{4\pi}}_{=:B_{i,j}^\epsilon},
 \end{equation}
From \eqref{ineq_int}, it follows that $A_i^\epsilon\geq 0$ for all $i =1,\dots,n+m$, with equality achieved if and only if each $D^\epsilon_i$ is a disk or an annulus. Note that $B_{i,j}^\epsilon\geq 0$ as well for all $i$ and $j$, since for any $i\neq j$, at most one of $i\prec j$ and $j\prec i$ can hold. Putting these together yields that
$\tilde I^\epsilon\ge 0$ for any sufficiently small $\epsilon>0$. 

In the rest of the proof, we assume $\Gamma$ is \textbf{not} a union of nested circles with constant $\gamma_i$'s on each connected component. Therefore at least one of the following 3 cases must be true. In each case we aim to show that $\tilde I_\epsilon\geq c_0>0$, where $c_0$ is independent of $\epsilon$ for all sufficiently small $\epsilon>0$.

 \textbf{Case 1.} There exists some open curve $\Gamma_i$ that is not a loop.  In this case $D^\epsilon_i$ is simply-connected, and   $p^\epsilon=0$ on $\partial D^\epsilon_i$ by \eqref{bdry_curve}. Applying Proposition~\ref{prop_p_curve} to $p^\epsilon$ in $D^\epsilon_i$, we have $\sup_{D^\epsilon_i} p^\epsilon \leq C\epsilon^2$, where $C$ is independent of $\epsilon$. This leads to $\int_{D^\epsilon_i} p^\epsilon dx \leq C\epsilon^3$, since $|D^\epsilon_i| = O(\epsilon)$ by  \eqref{P_nonzero}. As a result, for the index $i$ we have
 \[
 A_i ^\epsilon= \frac{|D^\epsilon_i|^2}{4\pi \epsilon^2}-\epsilon^{-2} \int_{D^\epsilon_i}p^\epsilon_idx \geq \frac{L_i^2}{4\pi}\left(\int_{S_i}\gamma_i(\alpha)d\alpha\right)^2 - C\epsilon,
 \] 
 where we again used  \eqref{P_nonzero} in the second inequality. This gives that $A_i^\epsilon \geq \frac{L_i^2}{8\pi}(\int_{S_i}\gamma_i(\alpha)d\alpha)^2 > 0$ for all sufficiently small $\epsilon>0$.

\textbf{Case 2.} There exists some closed curve $\Gamma_i$ that is either not a circle, or $\gamma_i$ is not a constant. In this case we aim to show that $A_i^\epsilon \geq c_0>0$, and this will be done by finding good approximations (independent of $\epsilon$) for both terms in $A_i^\epsilon$. 
For the first term $\frac{|D^\epsilon_i|^2}{4\pi \epsilon^2}$, using \eqref{P_nonzero} we again have
\begin{equation}\label{def_ji}
\frac{|D^\epsilon_i|^2}{4\pi \epsilon^2} \geq  \frac{L_i^2}{4\pi}\left(\int_{S_i}\gamma_i(\alpha)d\alpha\right)^2- C\epsilon =: J_i - C\epsilon,
\end{equation}
where $J_i>0$ is independent of $\epsilon$. For the second term $\epsilon^{-2}\int_{D_i^\epsilon} p_i^\epsilon dx$, rewriting the integral using the change of variables $x=R_i^\epsilon(a,\eta)$ gives
\[
\epsilon^{-2}\int_{D_i^\epsilon} p_i^\epsilon dx =  \int_{S_i}\int_{-1}^0 \frac{p^\epsilon(R_i^\epsilon(\alpha,\eta))}{\epsilon} \frac{\det(\nabla_{\alpha,\eta}R_i^\epsilon)}{\epsilon} d\eta d\alpha.
\]
Recall that in Proposition~\ref{P_pdecomposition} we defined $\tilde{p}^{\epsilon}(R^\epsilon_i(\alpha,\eta)) :=  c^\epsilon_i(1+\eta)$ and $q_\epsilon$ such that $p^\epsilon - \tilde p_\epsilon = q_\epsilon$. By \eqref{P_claim16} and \eqref{c_beta}, for all $\alpha\in S_i$ and $\eta\in(-1,0)$ we have
\[
\left|\frac{p^\epsilon(R_i^\epsilon(\alpha,\eta))}{\epsilon} - \beta_i (1+\eta)\right| \leq \left|\frac{ p^\epsilon(R_i^\epsilon(\alpha,\eta))}{\epsilon} - \frac{c_i^\epsilon}{\epsilon}(1+\eta)\right| + \left|\frac{c_i^\epsilon}{\epsilon}-\beta_i\right| \leq C\epsilon,
\]
where $\beta_i := \frac{2|U_i|}{L_i\int_{S_i}\gamma_i^{-1}(\alpha)d\alpha}$ is defined in \eqref{P_betadef2}. Combining this with the expression of the determinant in \eqref{eq_det}, we have
 \[
 \begin{split}
 \epsilon^{-2}\int_{D_i^\epsilon} p_i^\epsilon dx &=    \int_{S_i}\int_{-1}^0 (\beta_i(1+\eta) + O(\epsilon)) (L_i \gamma_i(\alpha) + O(\epsilon)) d\eta d\alpha\\
 &\leq \frac{|U_i|}{\int_{S_i} \gamma_i^{-1}(\alpha) d\alpha} \int_{S_i} \gamma_i(\alpha) d\alpha + C\epsilon =: K_i + C\epsilon,
 \end{split}
 \]
 where $K_i$ is independent of $\epsilon$.
Putting this together with \eqref{def_ji} yields the following:
\begin{equation}\label{A_ineq}\begin{split}
A_i^\epsilon &\geq  J_i - K_i - C\epsilon\\
&= \frac{L_i^2}{4\pi} \frac{\int_{S_i}\gamma_i(\alpha)d\alpha}{\int_{S_i}\gamma_i^{-1}(\alpha)d\alpha}\left(\int_{S_i}\gamma_i^{-1}(\alpha)d\alpha \int_{S_i}\gamma_i(\alpha)d\alpha - \frac{4\pi |U_i|}{L_i^2}\right) - C\epsilon.
\end{split}
\end{equation}
Let us take a closer look at the two terms inside the parenthesis. For the first term, Cauchy-Schwarz inequality gives
\[
\int_{S_i}\gamma_i^{-1}(\alpha)d\alpha \int_{S_i}\gamma_i(\alpha)d\alpha \geq 1,
\]
with equality achieved if and only if $\gamma_i$ is a constant. For the second term, the isoperimetric inequality yields
\[
\frac{4\pi |U_i|}{L_i^2} \leq 1,
\]
(recall that $L_i = |\partial U_i|$), with equality achieved if and only $U_i$ is a disk. By the assumption of Case 2, at least one of the inequalities must be strict, thus the parenthesis on the right hand side of \eqref{A_ineq} is strictly positive (and independent of $\epsilon$). Therefore there exists some constant $c_0>0$ such that $\tilde I^\epsilon \geq A_i^\epsilon \geq c_0$ for all sufficiently small $\epsilon$.

\textbf{Case 3.} There exist $i\neq j$ such that $i\nprec j$ and $j\nprec i$. Then it is clear that for such $i,j$, $B_{i,j}^\epsilon$ in \eqref{P_i} is given by $B_{i,j}^\epsilon = \frac{|D^\epsilon_{i}||D^\epsilon_{j}|}{4\pi \epsilon^2}$. Hence \eqref{P_nonzero} gives
\[
B_{i,j}^\epsilon \geq L_i L_j \bigg( \int_{S_i}\gamma_i(\alpha)d\alpha\bigg)\bigg( \int_{S_j} \gamma_j(\alpha)d\alpha\bigg) - C\epsilon,
\]
which yields $\tilde I^\epsilon \ge \frac{1}{2}L_i L_j ( \int_{S_i}\gamma_i d\alpha)( \int_{S_j} \gamma_j d\alpha)>0$ for all sufficiently small $\epsilon>0$. 

This finishes our discussion on all 3 cases. To conclude, since $\Gamma$ is not a union of nested circles with constant $\gamma_i$'s on each connected component, at least one of the 3 cases must hold, and all of them lead to $\tilde I^\epsilon\geq c_0>0$.
 \end{proof}

The above proposition immediately leads to the following corollary for the $\Omega<0$ case. 

\begin{corollary}\label{cor2_i} Assume that $\Gamma_i$ and $\gamma_i$ satisfy \textbf{\textup{(H1)}}--\textbf{\textup{(H3)}} for $i=1,\dots,n+m$.  Let $I^\epsilon$ be defined as in \eqref{def_i_eps}, and assume $\Omega< 0$. Then we have
$
I^\epsilon\geq 0
$ for all sufficiently small $\epsilon>0$. In addition, if $\Gamma$ is \textbf{not} a union of concentric circles all centered at the origin with constant $\gamma_i$'s, there exists some $c_0>0$ independent of $\epsilon$, such that $ I^\epsilon > c_0 > 0$ for all sufficiently small $\epsilon > 0$.
\end{corollary}

\begin{proof}

Let us decompose $I^\epsilon$ as follows (recall the definition of $\tilde I^\epsilon$ in \eqref{def_i_eps_tilde})
\begin{equation}I^\epsilon = \tilde I^\epsilon +   (-\Omega) \left( \epsilon^{-1} \int_{D^\epsilon} (|x|^2 + \nabla p^\epsilon \cdot x) dx\right) =: \tilde I^\epsilon + \underbrace{(-\Omega)}_{>0} J^\epsilon. \label{J_2}
\end{equation}
Recall that Proposition~\ref{prop2_i} gives $\tilde I_\epsilon\geq c_0>0$ as long as $\Gamma$ is not a union of nested circles with constant $\gamma_i$'s. By \cite[Lemma 2.11]{GomezSerrano-Park-Shi-Yao:radial-symmetry-stationary-solutions}, we have 
\[
\int_{D_i^\epsilon} (|x|^2 + \nabla p^\epsilon \cdot x) dx\geq 0 \quad\text{ for any $i=1,\dots,n+m$,}
\] thus $J^\epsilon \geq 0$.
  Putting them together, and using the fact that $\Omega<0$, we know $I^\epsilon\geq c_0>0$ if $\Gamma$ is not a union of nested circles with constant $\gamma_i$'s. 

To finish the proof, we only need to focus on the case that  the  $\Gamma_i$'s are nested circles with constant $\gamma_i$'s, but not all of them are centered at the origin. Assume that there exists $k\in\{1,\dots,n\}$ such that $\Gamma_k$ is a circle with radius $r_k$ centered at $x_k \neq 0$. Since $\gamma_k$ is a constant,  $D_k^\epsilon$ is an annulus given by $B(x_k,r_k+\epsilon\gamma_k)\setminus B(x_k,r_k)$. The symmetry of $D_k^\epsilon$ about $x_k$ immediately leads to $p^\epsilon|_{D_k^\epsilon} = -\frac{1}{2}|x-x_k|^2 +\frac{1}{2} (r_k+\epsilon\gamma_k)^2.$ An elementary computation gives
\[
\epsilon^{-1}\int_{D^\epsilon_k} (|x|^2 + \nabla p^\epsilon \cdot x) dx =\epsilon^{-1} \int_{D^\epsilon_k} |x|^2 - (x-x_k)\cdot x dx = \epsilon^{-1} |x_k|^2 |D_k^\epsilon| \geq  2\pi r_k \gamma_k  |x_k|^2>0,
\]
where in the second-to-last step we used that $|D_k^\epsilon| = 2\pi\epsilon r_k \gamma_k + \pi\epsilon^2 \gamma_k^2$. Setting $c_0:=2\pi r_k \gamma_k  |x_k|^2$ gives $I^\epsilon \geq c_0>0$, thus we can conclude.\end{proof}

Now we are ready to prove Theorem~\ref{thmA}. Note that for $\Omega<0$, the symmetry result  immediately follows from Proposition~\ref{prop1_i} and Corollary~\ref{cor2_i}. For $\Omega=0$, Proposition~\ref{prop1_i}--\ref{prop2_i} already imply that a stationary vortex sheet with positive strength must be a union of nested circles with constant strength on each of them. To finish the proof, we only need to show that these nested circles must be concentric.

\begin{proof}[\textbf{\textup{Proof of Theorem~\ref{thmA}}}]
For a uniformly-rotating vortex sheet with $\Omega<0$, the symmetry result for $\Omega<0$ is a direct consequence of Proposition~\ref{prop1_i} and Corollary~\ref{cor2_i}. Next we focus on the stationary (i.e. $\Omega=0$) case.

Combining Propsitions~\ref{prop1_i}--\ref{prop2_i}, we obtain that $\Gamma$ is a union of nested circles, and $\gamma_i$ is constant on $\Gamma_i$ for all $i=1\ldots,n$. It remains to show that all $\Gamma_i$'s are concentric. Let us denote by $\mathbf{v}_i $ the contribution to the velocity field by $\Gamma_i$. Since $\Gamma_i$ is a circle with constant strength $\gamma_i$, a quick application of the divergence theorem yields that $\mathbf{v}_i \equiv 0$ in the open disk enclosed by $\Gamma_i$, whereas $\mathbf{v}_i(x) = \dfrac{\gamma_i L_i (x-x^0_i)^\perp}{2\pi |x-x^0_i|^2}$ in the open set outside $\Gamma_i$, where $x^0_i$ is the center of the circle $\Gamma_i$.

Without loss of generality, let us reorder the indices such that $\Gamma_i$ is nested inside $\Gamma_j$ for $i<j$. Towards a contradiction, let $k>1$ be such that $\Gamma_k$ is the first circle that is not concentric with $\Gamma_1$. 
From the above discussion, we know that $\mathbf{v}_i=0$ on $\Gamma_k$ for $i =k+1,\dots,n$ (since $\Gamma_k$ is nested inside $\Gamma_i$),
whereas for $i=1,\dots,k-1$ we have $\mathbf{v}_i = \dfrac{\gamma_i L_i (x-x^0_1)^\perp}{2\pi |x-x^0_1|^2}$ on $\Gamma_k$, since all these $\Gamma_i$'s have the same center $x^0_1$ and are nested inside $\Gamma_k$. Summing them up (and also using the fact that $\Gamma_k$ contributes zero normal velocity on itself, since it is a circle with constant strength), we have
\[
BR(x) \cdot \mathbf{n} = \sum_{i=1}^n  \mathbf{v}_i(x) \cdot \mathbf{n}=  \left(\sum_{i=1}^{k-1} \gamma_i L_i\right)\frac{(x-x^0_1)^\perp \cdot \mathbf{n}}{2\pi |x-x^0_1|^2} \quad\text{ on }\Gamma_k,
\]
 where the right hand side is not a zero function since $\Gamma_k$ has a different center from $x_1^0$. This causes a contradiction with the fact that $\omega = \omega_0$ is stationary. As a result, all $\Gamma_1,\dots, \Gamma_n$ must be concentric circles, finishing the proof.
\end{proof}

\color{black}

\section*{Acknowledgements}

JGS was partially supported by the European Research Council through ERC-StG-852741-CAPA. JP was partially supported by NSF through Grants NSF DMS-1715418, and NSF CAREER Grant DMS-1846745. JS was partially supported by NSF through Grant NSF DMS-1700180. YY was partially supported by NSF through Grants NSF DMS-1715418, NSF CAREER Grant DMS-1846745, and Sloan Research Fellowship. JGS would like to thank Toan Nguyen for useful discussions.

\bibliographystyle{abbrv}
\bibliography{references}

\begin{tabular}{l}
\textbf{Javier G\'omez-Serrano}\\
{Department of Mathematics} \\
{Brown University} \\
{Kassar House, 151 Thayer St.} \\
{Providence, RI 02912, USA} \\ \\
{and} \\ \\ 
{Departament de Matem$\grave{a}$tiques i Inform$\grave{a}$tica} \\
{Universitat de Barcelona} \\
{Gran Via de les Corts Catalanes, 585} \\
{08007, Barcelona, Spain} \\
{Email: javier\_gomez\_serrano@brown.edu, jgomezserrano@ub.edu} \\ \\
\textbf{Jaemin Park} \\
{School of Mathematics, Georgia Tech} \\
{686 Cherry Street, Atlanta, GA 30332} \\
{Email: jpark776@gatech.edu} \\ \\
\textbf{Jia Shi}\\
{Department of Mathematics}\\
{Princeton University} \\
{409 Fine Hall, Washington Rd,}\\
{Princeton, NJ 08544, USA}\\
{e-mail: jiashi@math.princeton.edu}\\ \\
\textbf{Yao Yao} \\
 {School of Mathematics, Georgia Tech}\\
 {686 Cherry Street, Atlanta, GA 30332}\\
 {Email: yaoyao@math.gatech.edu}\\
\end{tabular}

\end{document}